\documentclass[12pt]{amsart}

\usepackage{amssymb}
\usepackage{amsfonts}
\usepackage{amsthm}
\usepackage{graphicx}
\usepackage{color}
\usepackage{url}
\usepackage{microtype}
\usepackage{tikz}
\usepackage{hyperref}
\definecolor{darkblue}{RGB}{0,0,160}
\hypersetup{
colorlinks,%
citecolor=black,%
filecolor=black,%
linkcolor=darkblue,%
urlcolor=darkblue
}

\usepackage[utf8]{inputenc}

\newcommand{\excise}[1]{}

\newtheorem{thm}{Theorem}[section]
\newtheorem{lemma}[thm]{Lemma}

\newtheorem{prop}[thm]{Proposition}
\newtheorem{conj}[thm]{Conjecture}

\newtheorem{question}[thm]{Question}

\theoremstyle{definition}
\newtheorem{example}[thm]{Example}
\newtheorem{remark}[thm]{Remark}
\newtheorem{defn}[thm]{Definition}

\numberwithin{equation}{section}

\newcommand{\ring}[1]{\ensuremath{\mathbb{#1}}}

\newcommand{\Itor}{I_{\ZZ\cB}}

\renewcommand\>{\rangle}
\newcommand\<{\langle}

\newcommand\1{\mathbf{1}}
\newcommand\2{\mathbf{2}}

\newcommand\NN{\ring{N}}

\newcommand\RR{\ring{R}}

\newcommand\ZZ{\ring{Z}}

\newcommand\kk{\Bbbk}
\newcommand\gl{\mathrm{gl}}

\newcommand\cB{{\mathcal B}}
\newcommand\cC{{\mathcal C}}
\newcommand\cX{{\mathcal X}}
\newcommand\cL{{\mathcal L}}
\newcommand\cM{{\mathcal M}}

\newcommand\fS{{\mathfrak S}}

\def\ol#1{{\overline {#1}}}

\newcommand\set[1]{\{#1\}}
\newcommand\abs[1]{|#1|}

\DeclareMathOperator\rank{rank} 
\newcommand{\defas}{\mathrel{\mathop{:}}=}   

\newcommand\Perp{\protect\mathpalette{\protect\independenT}{\perp}}
\def\independenT#1#2{\mathrel{\rlap{$#1#2$}\mkern2mu{#1#2}}}
\newcommand{\ind}[2]{\left.#1 \Perp #2 \inD}
\newcommand{\inD}[1][\relax]{\def\argone{#1}\def\temprelax{\relax}
  \ifx\argone\temprelax\right.\else\,\middle|#1\right.{}\fi}

\begin{document}

\title{Positive margins and primary decomposition}

\author{Thomas Kahle}
\address{Forschungsinstitut Mathematik \\ Z\"urich, Switzerland} 
\email{thomas.kahle@math.ethz.ch}
\author{Johannes Rauh}
\address{MPI MIS \\ Leipzig, Germany}
\email{rauh@mis.mpg.de}
\author{Seth Sullivant}
\address{Department of Mathematics  \\ North Carolina State University, Raleigh, NC}
\email{smsulli2@ncsu.edu}

\thanks{In November 2011 TK and JR worked intensively on this project
during a two-week ``Research in Pairs'' stay at Mathematisches
Forschungsinstitut Oberwolfach.  TK is supported by an EPDI
fellowship.  JR is supported by the Volkswagen Foundation.  SS is
support by the David and Lucille Packard Foundation and the
U.S.~National Science Foundation (DMS 0954865)}

\date{June 20, 2012}

\makeatletter
  \@namedef{subjclassname@2010}{\textup{2010} Mathematics Subject Classification}
\makeatother





\begin{abstract}
We study random walks on contingency tables with fixed marginals,
corresponding to a (log-linear) hierarchical model.  If the set of
allowed moves is not a Markov basis, then there exist tables with the
same marginals that are not connected.  We study linear conditions on
the values of the marginals that ensure that all tables in a given
fiber are connected.  We show that many graphical models have the
positive margins property, which says that all fibers with strictly
positive marginals are connected by the quadratic moves that
correspond to conditional independence statements.  The property
persists under natural operations such as gluing along cliques, but we
also construct examples of graphical models not enjoying this
property.   We also provide a negative answer to a question of
Engstr\"om, Kahle, and Sullivant by demonstrating that 
the global Markov ideal of the complete bipartite graph $K_{3,3}$ is
not radical.

Our analysis of the positive margins property depends on computing the
primary decomposition of the associated conditional independence
ideal.  The main technical results of the paper are primary
decompositions of the conditional independence ideals of graphical
models of the $N$-cycle and the complete bipartite graph $K_{2,N-2}$,
with various restrictions on the size of the nodes.
\end{abstract}

\maketitle
\tableofcontents

\section{Introduction}
\label{sec:introduction}

Let $\cB$ be a finite subset of $\ZZ^n$, and consider the graph with
vertex set $\NN^n$ (here, $\NN$ denotes the natural numbers including
zero) and edges $(u,v)$ whenever $u-v\in\pm\cB$.  We want to study the
connected components of this graph.  Our motivation comes from Markov
chain random walks on $\NN^n$ using the elements in~$\cB$ as moves.
If every edge in this graph has positive probability, then the
connected components are the irreducible components of
the Markov chain.

A necessary condition for $u,v\in\NN^n$ to be connected by $\cB$ is
that their difference vector $u-v$ lies in the lattice $\ZZ\cB$
generated by~$\cB$.  We want to know when this condition is
sufficient.  In this paper we assume that the lattice $\ZZ\cB$ is
\emph{saturated}, that is, it can be written as the integer kernel
$\ker_\ZZ A$ of an integer matrix~$A$.  We do not require
$\cB$ to be a \emph{basis} of $\ZZ \cB$---there can be more than
$\rank(\ZZ\cB)$ generators.
For any $u\in\NN^n$ we call $(u+\ZZ\cB)\cap\NN^n$ the \emph{fiber} of~$u$.
For example, in the statistical analysis of contingency tables, people
are interested in the set of all contingency tables with given
marginals.  In this case, the matrix $A$ corresponds to the linear map
that computes the marginals from a contingency table.  Monte Carlo
sampling techniques are then applied to compute approximate p-values
in Fisher's exact test for conditional
inference~\cite{diaconissturmfels98,drton09:_lectur_algeb_statis}.

In the literature, often the following inverse problem is studied:
Given a saturated lattice and a point $u\in\NN^n$, find a set $\cB$
such that the fiber of $u$ is connected.  Such a set is called a
\emph{Markov subbasis} in~\cite{yuguo06:_sequen}.  Ideally one wants
to compute a \emph{Markov basis}, a finite set that connects
\emph{all} fibers at once.

The fundamental theorem of Markov bases
(see~\cite[Theorem~1.3.6]{drton09:_lectur_algeb_statis} and
Theorem~\ref{thm:fundamental-Markov-bases} below) implies that Markov
bases can be found using computer algebra.  Despite fast computers,
excellent algorithms~\cite{malkin-2006}, and efficient
implementations~\cite{4ti2}, computing Markov bases remains hard and
is out of reach for many practical applications.  Furthermore, since
Markov bases are guaranteed to connect every fiber, they might be much
larger than needed to connect a particular fixed given fiber.  In
this paper we study conditions on the fiber that certify that a given
set of moves connects this fiber.  In particular, we say that $\cB
\subseteq \ker_{\ZZ}A$ has the positive margins property with respect
to the matrix $A$ if $(Au)_{i} > 0$ for all $i$ implies that the fiber
of $u$ is connected by~$\cB$.  
This property depends not only on the lattice $\ker_{\ZZ} A$ but also
the particular matrix~$A$.

The main focus in this paper is on lattices and moves associated to
graphical models.  For graphical models there is a canonical set of
``simple'' moves, which correspond to the global Markov conditional
independence statements.  It has been observed that for some models,
if a contingency table $u$ has strictly positive margins, then these
simple moves connect the fiber of~$u$~\cite{chen10:_markov}.
Similarly, for the no-three-way interaction model Bunea and Besag
proved a positive margins property for a set of ``simple''
moves~\cite{bunea00:_mcmc_i_j_k}.  In the present paper, we perform a
systematic study of the positive margins property for graphical
models.

Connectivity of lattice walks can be studied with tools from
commutative algebra using the following idea: Consider the binomial
ideal
\begin{equation*}
I_\cB = \< p^{u_+} - p^{u_-} : u\in\cB \> \subseteq \kk[p_1,\dots,p_n],
\end{equation*}
where $\kk$ is a field, $u = u^{+} - u^{-}$ is the minimal support
decomposition of $u$ into positive and negative parts, and $p^{v} =
p_{1}^{v_{1}} \cdots p_{n}^{v_{n}}$.  The following result is
well-known (see~\cite{diaconis98:_lattic} and references therein):
\begin{prop}
\label{prop:DES-criterion}
Two points $u,v\in\NN^n$ are connected by a path
$u=u_0,u_1,\dots,u_n=v$ in $\NN^n$ with $u_{i+1}-u_i\in\pm\cB$ if and
only if $p^u-p^v\in I_\cB$.
\end{prop}

Diaconis, Eisenbud, and Sturmfels~\cite{diaconis98:_lattic} proposed
to analyze the connectivity of the fibers of $\cB$ using a primary
decomposition of the ideal~$I_{\cB}$.  In Section~\ref{sec:background}
we study the positive margins property and relate it to decompositions
of~$I_\cB$.  In particular, Lemma~\ref{lem:positive-margins} gives a
sufficient condition and a necessary condition for the positive
margins property to hold.  We also study a generalization of the
positive margins property, which we call interior point property.

Our ideals $I_{\cB}$ are conditional independence ideals, in addition
to being motivated by the application to random walks.  Primary
decompositions of conditional independence ideals are interesting in
their own right, since they reveal important information about the set
of probability distributions that satisfy the conditional independence
statements~\cite{drton09:_lectur_algeb_statis,HHHKR10:Binomial_Edge_Ideals,RauhAy2011:Robustness_and_CI_ideals,swanson-taylor11_CI}.
Moreover, it is interesting to know whether $I_{\cB}$ is radical.
Section~\ref{sec:graphicalmodels} provides background on graphical
models and conditional independence.

Section~\ref{sec:pos-margins-and-graphs} studies the positive margins
property of graphical models and radicality of their global Markov
ideals.  Both properties are preserved when forming the coned graph
and when gluing graphs along cliques.  In particular, decomposable
graphs have both properties.  From our results we deduce that, if all
nodes are binary, then all global Markov ideals of graphs on five or
fewer nodes are radical, and the complete bipartite graph $K_{2,3}$ is
the only graph on five or fewer nodes which does not satisfy the
original positive margins property while it does satisfy the interior
point property.  We also find graphical models without positive
margins property, for any choice of matrix $A$, when the contingency
table is sufficiently large (Theorem~\ref{thm:orth-lat-squares}).

The graphical models of the $N$-cycle and the complete bipartite graph
$K_{2,N-2}$ (with restrictions on the sizes of the contingency tables)
are discussed in detail in Sections~\ref{sec:binary-n-cycles}
and~\ref{sec:K2n}.  We construct Markov bases and show that the global
Markov ideals are radical by computing the primary decompositions.

Our results suggest a number of different directions for further
research.  First, in our analysis we profited from the fact that all
conditional independence ideals that we studied are radical.  There
do exist non-radical global Markov ideals, but we do not know how
abundant those are. 
Second, our proofs of the positive margins property and interior point
property for~$\cB$ depend on knowledge of a Markov basis for the
lattice~$\ZZ\cB$.  It remains an open problem to develop proofs that
do not depend on that extra knowledge.


\section{Lattice walks, binomial ideals, and positive margins}
\label{sec:background}

As Diaconis, Eisenbud, and Sturmfels \cite{diaconis98:_lattic}
observed, the connectivity of the lattice walk induced by the moves in
$\cB$ can be analyzed by looking at a decomposition of the
ideal~$I_{\cB}$.  Indeed, suppose that $I_{\cB} = \cap_{i} I_{i}$.
Then $p^u-p^v\in I_\cB$ if and only if $p^u-p^v\in I_i$ for all~$i$.
The following example demonstrates how to profit from this simple
idea.

\begin{example}[{cf. \cite[Example~1.2]{diaconis98:_lattic}}]
\label{e:simple}
Which lattice points in $\NN^2$ can be connected by the moves
$\cB=\{(2,-2), (3,-3)\}$?  The solution can be read off from the
decomposition
\begin{equation*}
I_\cB = \<p_1^2-p_2^2, p_1^3-p_2^3\> = \<p_1-p_2\> \cap \<p_1^2,p_2^2\>.
\end{equation*}
Now, $p^a-p^b\in\<p_1-p_2\>$ if and only if $a_1+a_2=b_1+b_2$, while
$p^a-p^b\in \<p_1^2, p_2^2\>$ if and only if $\max\{a_1,a_2\}\ge 2$
and $\max\{b_1,b_2\}\ge 2$.  Hence, $a$ and $b$ are connected by $\cB$
if and only if $a_1+a_2=b_1+b_2$ and
$\min\{\max\{a_1,a_2\},\max\{b_1,b_2\}\}\ge 2$. \qed
\end{example}

The first decomposition that comes to mind is primary decomposition.
If the ground field is algebraically closed, then, since $I_\cB$ is
binomial, there is a binomial primary decomposition $I_\cB=\cap_iP_i$,
where $P_i$ are generated by binomials.  When the primary
decomposition introduces new coefficients, then it is too fine to
accurately reflect the combinatorics of~$\mathcal{B}$---everything
that matters in Proposition~\ref{prop:DES-criterion} are pure
differences (i.e.~binomials of the form $p^u-p^v$).  In this case one
should work with a \emph{mesoprimary decomposition}
of~$I_\cB$~\cite{kahle11mesoprimary}, the finest decomposition into
unital binomial ideals (i.e.~ideals generated by pure differences and
monomials).  In the examples studied in this paper all ideals
$I_{\cB}$ are radical, and the primary and mesoprimary decompositions
agree.

The ``most important'' associated prime of $I_\cB$, according
to~\cite[p.\ 116]{sturmfels02:_solvin_system_polyn_equat}, is the
\emph{toric ideal}
$\Itor=I_{\cB} : \big(\prod_{i \in [n]}p_{i}\big)^{\infty}$,
which is the only associated prime of $I_\cB$ that does not contain variables.
It equals the kernel of the ring homomorphism
\begin{equation}\label{eq:toricMorphism}
\phi_A^* : \kk[p_i:i=1,\dots,n] \to \kk[\theta_j,\theta_j^{-1}:j=1,\dots,h],\quad
p_i\mapsto \prod_j \theta_j^{A_{j,i}}
\end{equation}
where $A$ is an integral matrix such that $\ker_{\ZZ}A = \ZZ \cB$.
Equivalently,  $\Itor = \< p^{u} -
p^{v}:u, v \in \NN^{n }, Au = Av \>$~\cite{sturmfels96:_gr_obner_bases_and_convex_polyt}.
From this follows:
\begin{thm}[Fundamental Theorem of Markov bases~{\cite[Theorem~1.3.6]{drton09:_lectur_algeb_statis}}]
\label{thm:fundamental-Markov-bases} A set
$\cB \subseteq \ker_\ZZ A$ is a Markov basis  if and only if $I_{\cB} = \Itor$.
\end{thm}

The following is our basic definition.
\begin{defn}
\label{def:pos-margins}
Assume that $\cB$ generates a saturated lattice $\ZZ\cB$, and let $A$
be a non-negative integer matrix such that $\ZZ\cB \subseteq \ker_\ZZ A$.
Then $\cB$
has the \emph{positive margins property} (with respect to $A$) if
$(Au)_{i} > 0$ for all~$i$ implies that the fiber of $u$ is connected.
\end{defn}

In most of the examples below, $\ZZ\cB=\ker_\ZZ A$.
Still, the choice of the matrix $A$ is crucial.
In many situations there
is a canonical choice, such as the marginal computing matrix in the
case of graphical models (see Section~\ref{sec:graphicalmodels}).  We
can augment any matrix by adding rows which do not effect
$\ker_{\ZZ}A$, but yield further nontrivial positivity conditions to
check.  A natural choice is to add all linear functionals
corresponding to facets of the cone $\RR_{\geq 0}A$ generated by the
columns of $A$.  In this case, the condition $(Au)_i>0$ for all $i$
says that $Au$ lies in the relative interior of the
cone~$\RR_{\geq 0}A$.

\begin{defn}
\label{def:intPoint}
Let $\cB$ be a set of generators of the integer kernel $\ker_\ZZ A$ of the
integer matrix~$A$.  Then $\cB$ has the \emph{interior point property}
if it connects every fiber for which $Au$ lies in the relative interior of the
cone~$\RR_{\geq 0}A$.
\end{defn}

We now prove an algebraic criterion to decide the
positive margins property.  For any ideal $I\subseteq\kk[p]$, let
$m_I:=\prod \set{p_i \notin I}$ be the product of the variables not
contained in $I$, and let $u_I$ be the exponent vector of~$m_I$.  We
also need the product $\hat m_I:=\prod \set{p_i : (I:p_i) = I}$ of all
variables that are regular modulo $I$ and its exponent vector $\hat
u_I$.  If $I$ is a prime or a radical cellular ideal, then $m_I=\hat
m_I$.
\begin{lemma}
\label{lem:positive-margins}
Let $\cB$ span a saturated sublattice of $\ker_\ZZ A$ for some
non-negative integer matrix $A$.  Let $I_{\cB}=(\cap_{i=1}^c I_i)\cap
\Itor$ be a decomposition such that $\Itor\not\subseteq I_i$ for
all~$i$.
\begin{itemize}
\item If for all $i=1,\dots,c$ there exists $j$ such that $(A
u_{I_i})_j = 0$, then $\cB$ has the positive margins property with respect to $A$.
\item If $\cB$ has the positive margins property with respect to $A$, then for all
$i=1,\dots,c$ there exists $j$ such that $(A \hat u_{I_i})_j=0$.
\end{itemize}
\end{lemma}
\begin{proof}
For the first statement, suppose that $u,v\in\NN^n$ lie in the same
fiber, but are not connected.  Then $p^u-p^v\in \Itor\setminus
I_{\cB}$, and hence $p^u-p^v\notin I_i$ for some~$i$.  In particular,
either $p^u\notin I_i$ or $p^v\notin I_i$.  Assume that we are in the
first case.  Then $p^u$ is a divisor of $m_{I_i}^a$ for some
integer~$a$.  Now, if there exists $j$ such that $(A u_{I_i})_j=0$,
then also $(A u)_j = 0$, since $A$ is non-negative.  This shows the
first statement.

For the second statement, suppose that $(A\hat u_{I_i})_j >0$ for some
$i$ and all $j$.  Let $p^u-p^v$ be a binomial in $\Itor\setminus
I_i$.  Then $(A(u + \hat u_{I_i})) = (A(v + \hat u_{I_i})) > 0$,
but since $\hat m_{I_i}(p^u - p^v)\notin I_i$, the two vectors $u+\hat
u_i$ and $v+\hat u_i$ are not connected by~$\cB$.
\end{proof}

Note the asymmetry between the two directions, the first using
$u_{I_i}$, the second~$\hat{u}_{I_i}$.  If all $I_i$ are prime, then
$m_{I_i} = \hat{m}_{I_i}$.  In this case
Lemma~\ref{lem:positive-margins} gives an equivalent characterization
of the positive margins property.

If the positive margins property is not satisfied, then one might
still hope that the fibers are connected if the marginals are large
enough.  This is the case in Example~\ref{e:simple}.  Unfortunately,
if $I_\cB$ is radical, then this is not true:
\begin{lemma}
\label{lem:strong-pos-marg-radical}
Assume that $\cB$ does not have the positive margins property with
respect to $A$, and suppose that $I_{\cB}$ is radical.  For any $b>0$
there exist $u,v\in\NN^n$ such that $(A u)_j = (A v)_j \ge b$ for all
$j$, but $p^u-p^v\not\in I_{\cB}$.
\end{lemma}
\begin{proof}
Let $I_{\cB} = \Itor \cap(\cap_i P_i)$ be the decomposition into
minimal primes.  By assumption and Lemma~\ref{lem:positive-margins},
for some $i$ the vector $u_{P_i}$ satisfies $(A u_{P_i})_j>0$ for
all~$j$.  For any binomial $p^u-p^v\in \Itor\setminus P_i$ there
exists a $c$ large enough such that the exponents satisfy
$(A(u+cu_{P_i}))_j = (A(v+cu_{P_i}))_j \geq b$ for all $j$.  Since
$P_i$ is prime, $m_{P_i}$ is regular and therefore
$m_{P_i}^c(p^u-p^v)\notin P_i$.  Hence $u+cu_i$ and $v+cu_i$ are not
connected.
\end{proof}

Example~\ref{e:simple} shows that the radicality assumption in
Lemma~\ref{lem:strong-pos-marg-radical} is necessary.


\section{Graphical models and the global Markov statements}
\label{sec:graphicalmodels}

Let $V=[N] := \{1,\dots,N\}$ for some integer $N>1$.  For each $v\in
V$ let $X_v$ be a discrete random variable taking values in $[d_v]$, $d_v\ge
2$.
Let $d = (d_{v})_{v \in V }$ and
let $\cX = \prod_{v \in V} [d_{v}]$.  For any $W \subseteq V$ 
the random vector $X_W = (X_v)_{v\in W}$ takes values in $\cX_{W} = \prod_{v \in W} [d_{v}]$.
If $x \in \cX$ and $W \subseteq V$, let
$x_{W} := (x_{v})_{v \in W}$.  With $h=|\cX|$, denote $\RR^{h} :=
\bigotimes_{v \in V} \RR^{d_{v}}$ the space of real $d_{v_{1}} \times
\cdots \times d_{v_{r}}$ arrays
of the form $p = (p_{x})_{x \in \cX}$.  Then $\RR^h$ contains the
probability simplex
\[
\Delta_{h-1} := \left\{ p \in \RR^{h} :  \sum_{x \in \cX} p_{x} = 1, p_{x} \geq 0 \mbox{ for all }
x \in \cX \right\}.
\]
Each $p \in \Delta_{h-1}$ represents a joint probability
distribution of $(X_v)_{v\in V}$.
The dependencies among $X_1,\dots,X_N$ are often visualized 
by an undirected graph $G = (V,E)$.
In this paper, all graphs are undirected and simple.
There are two ways that such a graph can be interpreted as a
statistical model, i.e.~as a family of joint probability
distributions.  The first leads to the global Markov model, the second
to the graphical model.

The \emph{global Markov model} associates to $G$ a family of
conditional independence statements among the random variables.  Let
$V=A\cup B\cup C$ be a partition of $V$ (into disjoint possibly empty
sets), and let $p\in\Delta_{h-1}$.  We write $\ind{X_{A}}{X_{B}}[X_{C}]
$ and say that \emph{ $X_A$ is independent of $X_B$ given
$X_C$} if and only if
\begin{equation*}
p_{x_A^{\phantom{\prime}}x_B^{\phantom{\prime}}x_C^{\phantom{\prime}}}p_{x_A'x_B'x_C^{\phantom{\prime}}} - p_{x_A^{\phantom{\prime}}x_B'x_C}p_{x_A'x_B^{\phantom{\prime}}x_C^{\phantom{\prime}}} = 0
\end{equation*}
for all possible values $x_A^{\phantom{\prime}},x_A',x_B^{\phantom{\prime}},x_B',x_C^{\phantom{\prime}}$ of $X_A,X_B,X_C$,
respectively.  See~\cite{drton09:_lectur_algeb_statis} for an
introduction to conditional independence from an algebraic point of
view.

For each $x_{c} \in \cX_{C}$ we construct a matrix $P^{A,B,x_{C}}$ of
format $\abs{\cX_A}\times \abs{\cX_B}$, with columns indexed by
$\cX_{A}$ and rows indexed by~$\cX_{B}$.  The entry in the $x_{A},
x_{B}$ position of $P^{A,B,x_{C}}$ is the
probability~$p_{x_{A}x_{B}x_{C}}$.  The conditional independence
statement $\ind{X_{A}}{X_{B}}|X_{C}$ is equivalent to the condition
that for all $x_{C} \in \cX_{C}$, $\rank(P^{A,B,x_{C}}) \leq 1$.  If
$C = \emptyset$ we get one matrix, and in general we get $|\cX_C|$
matrices.

Let $I_{\ind{X_{A}}{X_{B}}[X_{C}]}$ be the ideal in $\RR[p_x:x\in\cX]$
generated by the $2 \times 2$ minors of all the matrices
$P^{A,B,x_{C}}$.  If $\cC$ is a collection of conditional independence
statements, we let
\[I_{\cC} =  \sum_{(\ind{X_{A}}{X_{B}}|X_{C}) \in \cC}  I_{\ind{X_{A}}{X_{B}}|X_{C}}.\]
To the graph $G$ we associate the \emph{global Markov statements}
\[
\gl(G) =  \{ \ind{X_{A}}{X_{B}}[X_{C}] :  C \mbox{ separates } A \mbox{ and } B \mbox{ in } G, A\cup B\cup C = V \}.
\]
Separation means that every path in $G$ from some vertex $a \in A$ to
some vertex $b \in B$ traverses some vertex $c \in C$.
The \emph{global Markov model} of $G$ is the intersection of
$\Delta_{h-1}$ and the variety of $I_{\gl(G)}$; i.e.\ it
consists of all joint probability distributions
satisfying~$\gl(G)$.
Note that, while most statements in this paper are independent of the
choice of the field~$\kk$, only the variety over the real numbers has
a natural statistical interpretation.
In general, conditional independence
statements are defined for arbitrary subsets $A,B,C\subseteq V$,  and
the global Markov statements are defined without the requirement $A\cup B\cup C = V$.
However, if $A,B,C\subseteq V$ are disjoint subsets such that
$A\cup B\cup C\neq V$ and such that $C$ separates $A$ and $B$, then
the statement $\ind{X_A}{X_B}[X_C]$ is implied by the statements in
$\gl(G)$, see~\cite[Lemma~7.10]{TFP-II}.

Graphical models are defined parametrically: Let $\cC(G)$ be the set
of cliques of~$G$, where a clique is a set of vertices $W \subseteq V$
such that if $v_{1}, v_{2} \in W$, $v_1\neq v_2$, then $(v_{1},v_{2})$ is an edge
of~$G$.
To each clique $C \in \cC(G)$ and each $x_{C} \in \cX_{C}$ associate a
parameter $\theta^{C}_{x_{C}}$ (or an indeterminate, depending on the
context).  Let $\theta^{C} := (\theta_{x_{C}}^C)_{x_{C} \in \cX_{C}}$.
The image of the polynomial map
\begin{equation*}
\phi_{G}:  \bigoplus_{C \in \cC(G)} \RR^{d_{C}}  \rightarrow  \RR^{h},\quad
\phi_{G,x} (\theta^{C_{1}}, \ldots, \theta^{C_{r}}) =  
\prod_{C \in \cC(G)}  \theta^{C}_{x_{C}},
\end{equation*}
intersected with the probability simplex $\Delta_{h-1}$ is the
parametrized graphical model~$\cM^{*}_{G}$.
In other words, $\cM_G^\ast$ consists of all probability
distributions $p$ whose components can be written as a product of the
form
$p_x = \prod_{W\in\cC}f_W(x)$,
where $f_W$ are nonnegative functions that only depend on $x_v$ for
$v\in W$.  See~\cite{lauritzen96} for more about graphical models.

The map $\phi_G$ induces the ring homomorphism
\begin{equation*}
\phi_{G}^*:  \RR[p_{x} : x \in \cX] \rightarrow 
\RR[\theta^{C}_{y_{C}} : C \in \cC(G), y_{C} \in \cX_{C} ],\quad
p_{x}  \mapsto  \prod_{C \in \cC(G)} \theta^{C}_{x_{C}},
\end{equation*}
and its kernel $I_{G} = \ker \phi_{G}^{*}$
is the vanishing ideal of the image.  Then $\cM_{G} = V(I_{G}) \cap
\Delta_{h-1}$ is the closure of the parametrized graphical model
$\cM^{*}_{G}$.  We call $\cM_G$ the \emph{graphical model} of~$G$.
Note that other authors use the term ``graphical model'' only for the
set of strictly positive probability distributions in~$\cM_G$.

The ring homomorphism $\phi_G^*$ is of the
form~\eqref{eq:toricMorphism}; hence $I_G$ is a toric ideal.
The corresponding matrix $A_G$ has a natural interpretation: If $p$ is
a joint probability distribution of $(X_v)_{v\in V}$, then the product
$A_Gp$ contains, as subvectors, the marginal distribution induced by
$p$ on any clique of $G$.  This collection of marginals are 
the \emph{G-marginals} of~$p$.  The cone generated by the columns of
$A_G$ is known as the \emph{marginal cone}.

It is easy to check that the graphical model is a subset of the global
Markov model.  Moreover, the Hammersley-Clifford Theorem \cite{Besag74} says that if
a probability distribution is strictly positive (that is $p_x > 0$ for
all $x$ in the state space), then $p$ lies in the graphical model if
and only if $p$ lies in the global Markov model.
Algebraically, this theorem says that $I_G$ equals the toric component
of~$I_{\gl(G)}$.

In general, $I_{\gl(G)}\subsetneq I_G$, in which case, there may be
probability distributions which satisfy the conditional independence
statements $\gl(G)$, but are not in the closure of the graphical
model.  In fact, $I_{\gl(G)}=I_{G}$ if and only if $G$ is a chordal
graph~\cite{geigermeeksturmfels06}.  As suggested in \cite[Chapter
8]{sturmfels02:_solvin_system_polyn_equat} and
\cite{geigermeeksturmfels06}, the discrepancy between the two models
can be analyzed using primary decomposition.


\section{The positive margins property and graphical models}
\label{sec:pos-margins-and-graphs}

In this section we study which global Markov models have the positive
margins property.  Let $G$ be a graph with vertex set $V = [N]$, and
let $d=(d_v)_{v\in V}\in\NN^N$ with $d_v\ge 2$ for all~$v$.  We say
that $(G,d)$ has the \emph{positive margins property}, if the
quadratic moves $\cB_{\gl(G)}$ have the
positive margins property with respect to the canonical matrix~$A_G$,
 and $(G,d)$ has the \emph{interior point
property} if $\cB_{\gl(G)}$ has the interior point property.

Our main tool is Lemma~\ref{lem:positive-margins} which we translate 
here to graphical models.  As all global Markov ideals with known
primary decompositions are radical, we only formulate the radical
case.
\begin{lemma}
\label{lem:graphical-positive-margins}
Let $I_{\gl(G)}=(\cap_{i=1}^c P_i)\cap I_G$ be a
decomposition into prime ideals such that $I_G\not\subseteq P_i$ for
all~$i$.  Then $(G,d)$ has the positive margins property if and only
if for all $i=1,\dots,c$ the $G$-margins of $u_{P_i}$ are not strictly
positive.
\end{lemma}
Table~\ref{tab:results} summarizes some of our computational results.
We computed Markov bases with 4ti2~\cite{4ti2} and binomial primary
decompositions using the package
\texttt{Binomials}~\cite{kahle11:binom-jsag} in Macaulay2~\cite{M2}.
Then we used the Macaulay2 package
\texttt{Polyhedra}~\cite{birkner09:_polyh} to check the condition of
Lemma~\ref{lem:graphical-positive-margins} applied to the primary
decomposition.
\begin{table}[htb]
\centering
\begin{center}
\begin{tabular}{|l|c|c|c|c|}
\hline
graph & pos. margins & interior point & $I_{\gl(G)}$ radical & \# of min. primes \\
\hline
$C_{4}$ &  yes & yes & yes &  $9$ \\
square-pyramid & yes & yes & yes & $81$ \\
$G_{48}$ & yes & yes & yes & $201$ \\
$K_{2,3}$ & no & yes & yes & $37$ \\
$C_5$  & yes   & yes & yes & $41$ \\
\hline
\end{tabular}
\end{center}
\caption{Properties of binary graphical models for selected irreducible graphs.}
\label{tab:results}
\end{table}
The binary graphical model of every graph on five or fewer vertices
that is not mentioned in Table~\ref{tab:results} satisfies the
positive margins property, and the corresponding global Markov ideals
are radical. 

These results suggest two general questions:
\begin{itemize}
\item Is it true that for any graphical model the ideal $I_{\gl(G)}$
is radical~\cite{TFP-II}?
\item Does every graphical model have the interior point property?
\end{itemize}
The answers to both questions are negative in general.
Example~\ref{ex:K33} discusses the binary CI ideal of $K_{3,3}$ which
is not radical.  Theorem~\ref{thm:badexample} settles the second
question.

Before discussing the graphs of Table~\ref{tab:results}, we treat
reducible graphs.  Note that
all graphs on five or fewer vertices not contained in this
table are either complete or decomposable, in the following sense:
\begin{defn}
A graph $G=(V,E)$ is \emph{reducible} if there exist proper subsets
$V_1,V_2\subset V$ such that $V_1\cap V_2$ is a clique, and such that
$G$ is the union of the 
subgraphs $G_1$ and $G_2$ induced on $V_1$ and $V_2$.
Moreover, $G$ is \emph{decomposable} if $G_1$ and $G_2$ are complete
or decomposable.
\end{defn}

\begin{lemma}
\label{lem:TFP}
Let $\kk$ be algebraically closed.  Assume that $G$ is reducible into
$G_1=(V_1,E_1)$ and $G_2=(V_2,E_2)$.  If both $I_{\gl(G_1)}$ and
$I_{\gl(G_2)}$ are radical, then $I_{\gl(G)}$ is radical.
\end{lemma}
\begin{proof}
This is \cite[Corollary~7.13]{TFP-II} together with the observation
that the toric fiber product of prime ideals is a prime ideal.
\end{proof}

\begin{lemma}
\label{lem:reducible-posmarg}
If $G$ is reducible into $G_1=(V_1,E_1)$ and $G_2=(V_2,E_2)$ and if
both $(G_1,(d_v)_{v\in V_1})$ and $(G_2,(d_v)_{v\in V_2})$ have the
positive margins property, then $(G,d)$ also has the positive margins
property.
\end{lemma}
\begin{proof}
The proof is essentially the same as that of~\cite[Theorem~2.9]{sullivant07:_toric}, which
shows how to obtain a Markov basis of $G$ from Markov bases of $G_1$
and~$G_2$.  The fact that we do not have Markov bases here is
compensated by the fact that we do not want to connect all fibers, but
just those fibers with positive margins.  In order to apply the proof
of~\cite[Theorem~2.9]{sullivant07:_toric} two
things need to be checked: (1)~A fiber with positive $G$-margins restricts to fibers with
positive $G_1$-margins and $G_2$-margins, respectively.  (2)~When the
construction that turns Markov bases of $G_1$ and~$G_2$ into a Markov
basis of $G$ is applied to $\cB_{\gl(G_1)}$ and $\cB_{\gl(G_2)}$, then
the result is a subset of $\cB_{\gl(G)}$.  For brevity we omit the
details.
\end{proof}
By Lemmas~\ref{lem:TFP} and Lemma~\ref{lem:reducible-posmarg}, decomposable graphs
have the positive margins property and radical global Markov ideals (for all $d$).  On four or fewer
vertices there is only one graph, the four-cycle~$C_4$, which is
neither complete nor decomposable.  The following theorem is proved in
Section~\ref{sec:binary-n-cycles}.
\begin{thm}
\label{thm:Cn-pos-marg}
For $N\ge 4$ the binary $N$-cycle model has the positive margins
property.  Its global Markov ideal $I_{\gl(C_N)}$ is radical.
\end{thm}

On five vertices there are five irreducible graphs: The complete graph
(which trivially has the positive margins property), the five-cycle~$C_5$
(covered by Theorem~\ref{thm:Cn-pos-marg}), the complete
bipartite graph $K_{2,3}$, the square pyramid, and
the graph $G_{48}$ (see Fig.~\ref{fig:G48}; the name $G_{48}$ comes
from~\cite{ReadWilson98:Atlas_of_Graphs}).  The complete bipartite
graph $K_{2,3}$ is treated in the following theorem, proved in the end
of Section~\ref{sec:K2n}.
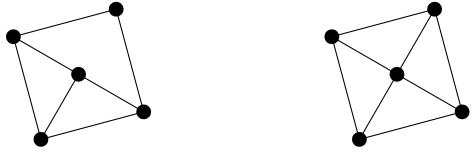
\begin{figure}
\centering
\begin{tikzpicture}[scale=1, rotate=-30]
  \path (1,1) coordinate (X1);
  \path (2,1) coordinate (X2);
  \path (3,1) coordinate (X3);
  \path (2,0) coordinate (X4);
  \path (2,2) coordinate (X5);
  \foreach \i in {1,...,5} { \fill (X\i) circle (2.8pt); }
  \draw (X1) -- (X5) -- (X3) -- (X4) -- (X2) -- (X1) -- (X4);
  \draw (X2) -- (X3);
\end{tikzpicture}
\hspace{2cm}
\begin{tikzpicture}[scale=1, rotate=-30]
  \path (1,1) coordinate (X1);
  \path (2,1) coordinate (X2);
  \path (3,1) coordinate (X3);
  \path (2,0) coordinate (X4);
  \path (2,2) coordinate (X5);
  \foreach \i in {1,...,5} { \fill (X\i) circle (2.8pt); }
  \draw (X1) -- (X5) -- (X3) -- (X4) -- (X2) -- (X1) -- (X4);
  \draw (X3) -- (X2) -- (X5);
\end{tikzpicture}
\caption{$G_{48}$ and the square pyramid}
\label{fig:G48}
\end{figure}

\begin{thm}
\label{thm:K22-pos-marg-int-point}
For $N\ge 4$, the complete bipartite graph $K_{2,N-2}$, where 
$d_{v} = 2$ for the first group of $2$ nodes, has the interior point property.  It
has the positive margins property if and only if $N=4$.
Its global Markov ideal is radical for all $N\ge 4$.
\end{thm}

We next discuss the pyramid.  To obtain a more general result the
following definition is needed: For any graph $G=(V,E)$ with vertex
set $V=[N]$, the \emph{cone} over $G$ is the graph $\hat G=(\hat
V,\hat E)$ with $\hat V=V\cup\{0\}$ and $\hat E=E\cup \{(0,i) :
i\in[N]\}$.
\begin{lemma}
\label{lem:pyramid-radical}
Assume that $\kk$ is a perfect field.
If $I_{\gl(G)}$ is radical for some $d\in\NN^V$, then $I_{\gl(\hat
G)}$ is radical for all $\hat d\in\NN^{\hat V}$ with $\hat d_v=d_v$
for all $v\in V$.
\end{lemma}
\begin{proof}
Let $\hat\cX=[d_0]\times\cX$.  For any polynomial
$f\in\kk[p_x:x\in\cX]$ denote by $\hat f_i$ the polynomial in
$\kk[p_y:y\in\hat\cX]$ where each variable $p_x$, $x\in\cX$, has been
replaced by $p_{ix}$.  Let $I_i$ be the ideal generated by the
polynomials $\hat f_i$ for all $f\in I_{\gl(G)}$.
The equality
\begin{equation*}
\gl(\hat G) = \left\{
\ind{X_{A}}{X_{B}}[X_{C}\cup X_0] : \ind{X_{A}}{X_{B}}[X_{C}]\in\gl(G)
\right\}
\end{equation*}
implies $I_{\gl(\hat G)}=I_1+\dots+I_{d_0}$.  The ideals
$I_1,\dots,I_{d_0}$ are radical, since $I_{\gl (G)}$ is radical, and
$I_1,\dots,I_{d_0}$ are generated by polynomials in disjoint sets of
variables.  To show that their sum is also radical it suffices to show
that the tensor product of reducible rings is again reducible.  This
is true if the field $\kk$ is perfect
by~\cite[Chapter~5,~§15]{Bourbaki50:Algebre_4-7}.
\end{proof}
\begin{lemma}
\label{lem:pyramid-lemma}
If $(G,d)$ has the positive margins property,
then $(\hat G,\hat d)$ also has the positive margins property, where
$\hat d_v=d_v$ for $v\in [N]$ and $\hat d_0$ is arbitrary.
\end{lemma}
\begin{proof}
Any contingency table $\hat u$ for $\hat G$ can be seen as a family
$(u^{(i)})_{i\in[\hat d_0]}$ of contingency tables for~$G$.  If $\hat
u$ has positive $\hat G$-margins, then each $u^{(i)}$ has positive
$G$-margins.  Now $\hat u$ and $\hat v$ have the same $\hat G$ margins
if and only if $u^{(i)}$ and $v^{(i)}$ have the same $G$-margins for
all~$i$.  Hence, if $\hat u$ and $\hat v$ have the same positive $\hat
G$-margins, then $u^{(i)}$ and $v^{(i)}$ are connected by quadratic
moves for all $i$, and the same moves can be used to connect $\hat u$
and~$\hat{v}$.
\end{proof}

It remains to discuss $G_{48}$.  It is easy to see that the binary
model for this graph is equal to the model of $K_{2,2}$ with
$d=(2,2,2,4)$, and therefore covered by
Theorem~\ref{thm:K22-pos-marg-int-point}---$G_{48}$ has the positive
margins property, and its global Markov ideal is radical.

Next, we give an example of a global Markov ideal that is not radical.
\newcommand{\Itt}{I_{\gl(K_{3,3})}}
\begin{example}
\label{ex:K33}
Consider the graph $K_{3,3}$, and let $d_v=2$ for all vertices $v\in
K_{3,3}$.  The global Markov ideal $I_{\gl(K_{3,3})}$ is contained in
a polynomial ring with 64 indeterminates.  It is generated
by $144=6\cdot 24$ quadrics corresponding to the six CI statements
$\ind{X_i}{X_{jk}}[X_{456}]$ and $\ind{X_i}{X_{jk}}[X_{123}]$, where
$\{i\}\cup \{jk\}$ runs through the non-trivial bipartitions of
$\{1,2,3\}$ and $\{4,5,6\}$, respectively.

The global Markov ideal $\Itt$ is complicated 
enough that Buchberger's algorithm for
Gr\"obner basis computation does not terminate within reasonable time.
On the other hand, a Gr\"obner basis of the graphical model
$I_{K_{3,3}}$ can be computed using 4ti2~\cite{4ti2}.  This is
another instance of the fact that toric ideals are less complex than
arbitrary binomial ideals~\cite{sturmfels91:_gr_bases_toric_variet}.

In view of these complications, the classical tools of computer
algebra do not work for this example, as they depend on Gr\"obner
bases.  However, we can use Proposition~\ref{prop:DES-criterion}:
containment of a binomial in a pure difference ideal can be checked by
analyzing the connected components of a graph.
We implemented this idea in a \verb!C++!-library that 
can test whether two exponent vectors lie in the same connected
component by enumerating their connected components via breadth-first
search.
The \verb!C++!-source code of our library is available on the Internet
under the GPL licence~\cite{graphbinomials}.  The directory
\verb!examples! contains code for $K_{3,3}$ and other graphs which
allows to generate the connected components and construct a path in
the case of connectivity.

To disprove radicality
it suffices to find a binomial $p^u-p^v\in I_{K_{3,3}}\setminus\Itt$
(for example, a degree four Markov move) and
a square-free monomial $p^w$ such that
$p^w(p^u-p^v)\notin \Itt$ while $p^{2w}(p^u-p^v)\in\Itt$.
Checking the degree four Markov moves $p^u-p^v$ and monomials of degree two, our program found the following witness: Let
\begin{equation*}
p^u-p^v :=
p_{121|222}p_{212|212}p_{122|112}p_{222|122}
- p_{221|222}p_{112|212}p_{222|112}p_{122|122},
\end{equation*}
and let $p^w:= p_{111|111}p_{221|111}$ (the vertical bar $|$ separates
the states of the two groups of nodes in $K_{3,3}$).  Then $p^{u}p^w$
and $p^{v}p^w$ are not connected by $\gl(K_{3,3})$, but $p^{u}p^{2w}$
and $p^{v}p^{2w}$ are connected.
The connected components of $p^{u+w}$ and $p^{v+w}$ consist of 18
monomials each, while that of $p^{u}p^{2w}$ and $p^{v}p^{2w}$ consists
of 90 monomials.  \qed
\end{example}

We now construct examples of graphical models that do not have the
interior points property (and, hence, cannot have any positive margins
property).  Remember that a graph $G$ is \emph{triangle-free} if it
does not contain a cycle of length three, and a graph is
\emph{two-connected} if it remains connected when a single node is
eliminated.
\begin{thm}
\label{thm:badexample}
Let $G$ be a two-connected triangle-free graph with $N$ vertices, and
let $p\geq N-1$ be a prime power.  If $d_{a} = p$ for all $a \in [N]$,
then $(G,d)$ does not have the interior point property.
\end{thm}

Before proving the theorem, we first give an explicit example.
\begin{example}
\label{ex:Seths-invidious-example}
Consider the four-cycle $C_4$ with $d = (3,3,3,3)$, and let
\[u = 
e_{1111} + e_{1222} + e_{1333} + 
e_{2123} + e_{2231} + e_{2312} +
e_{3132} + e_{3213} + e_{3321}.
\]
The marginal vector $A_{C_{4}} u$ of $u$ lies in the interior
of the marginal cone, and many other vectors with the same marginals
can be constructed by applying elements of the symmetry group
$(\ZZ/3\ZZ)^{4}$.  At the same time no quadratic move can be applied
to~$u$.
\qed
\end{example}

The combinatorially inclined reader may have observed two orthogonal
Latin squares of order three in the last two indices of the elements
contributing to $u$. 
Recall that a \emph{Latin square of order $d$} is a
$(d\times d)$-matrix $L$ with entries in $[d]$ such that each number
in $[d]$ appears exactly once in each row and in each column.  Two
Latin squares $L,L'$ are \emph{orthogonal} if $(L_{i,j},L'_{i,j}) =
(L_{k,l},L'_{k,l})$ implies $i=k$ and $j=l$.  For general $d$ the
number of mutually orthogonal Latin squares of order $d$ is not known.
The following is known:
\begin{enumerate}
\item There are at least 2 orthogonal Latin squares of order $d$,
unless $d\in\{1,2,6\}$.
\item There are at most $(d-1)$ orthogonal Latin squares of order~$d$.
\item If $d$ is a power of a prime, then there are precisely $(d-1)$
orthogonal Latin squares of order~$d$.
\end{enumerate}
See~\cite{DenesKeedwell74:Latin_squares_and_applications} for an
introduction Latin squares.  Theorem~\ref{thm:badexample} is a
corollary to these facts and the following theorem.
\begin{thm}
\label{thm:orth-lat-squares}
Let $G$ be a two-connected triangle-free graph with $N$ vertices.
If there exist $N-2$ mutually orthogonal Latin squares of order
$d_0\ge 2$, then $(G,(d_0,d_0,\dots, d_0))$ does not satisfy the
interior point property.
\end{thm}

\begin{proof}
Let $L^{(1)},\dots,L^{(N-2)}$ be mutually orthogonal Latin squares of
order $d_0$, and let
\begin{equation*}
\cL = \{ (i,j, L^{(1)}_{i,j},\dots, L^{(N-2)}_{i,j}) : i,j\in [d_0] \} 
\subset \cX = [d_0]^N.
\end{equation*}
The set $\cL$ has the property that for every pair $a,b \in [N]$,
$a\neq b$, one has
\begin{equation*}
\{(l_{a},l_{b}) : l \in\cL \} = [d_0]^{2}.
\end{equation*}
Since $G$ is triangle-free, all $G$-margins are $2$-way margins.
The vector $u(\cL)$ defined via
\begin{equation*}
u(\cL)(l) =
\begin{cases}
1, & \text{ if } l\in\cL,\\
0, & \text{ otherwise},
\end{cases}
\end{equation*}
has the following property: All entries in all its $2$-way margins are
ones.
The group $\fS_{d_0}^N$ ($N$-th direct power of the symmetric group
$\fS_{d_0}$ of $[d_0]$) acts on $\cX$ by permuting each factor.  This
action induces an action on the marginal cone that is transitive on
the extreme rays.  Under this action the margins of $u(\cL)$ are
invariant, which implies that $A_G u(\cL)$ lies in the interior of
this cone and, in particular, is not on any facet.

On the other hand, it is not possible to apply any quadratic global
Markov move to the table~$u(\cL)$.  Indeed, since $G$ is two
connected, any quadratic move $v$ corresponds to a statement
$\ind{X_{A}}{X_{B}}[X_{C}]$, where the separator $C$ contains at least
two distinct elements~$i,j$.  Hence, $v$ can only be applied to tables
where some entry in the $(i,j)$-marginal is two.  Therefore, $u(\cL)$
is isolated in its fiber.  On the other hand, the symmetric group
action on tables sends $u(\cL)$ to other points in its fiber, so that
the fiber is disconnected.
\end{proof}

\begin{example}[A binary grapical model without interior point
property]
\label{ex:binary-invidious}
We can use Theorem~\ref{thm:orth-lat-squares} to show that not all
graphs with binary nodes have the interior point property.  First, $C_4$ with
$d=(4,4,4,4)$ does not have the interior point property, since
there exists a pair of orthogonal Latin squares of order four.
We define a graph $G$ by splitting every vertex of $C_4$ into an edge
as in Figure~\ref{fig:splitC4}.
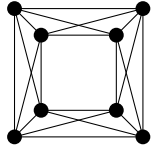
\begin{figure}[htb]
\centering
\begin{tikzpicture}[scale=1, 
]
\path (0,0) coordinate (X11); \path (0,0) ++(225:.5) coordinate (X12);
\path (0,1) coordinate (X21); \path (0,1) ++(135:.5) coordinate (X22);
\path (1,1) coordinate (X31); \path (1,1) ++(45:.5) coordinate (X32);
\path (1,0) coordinate (X41); \path (1,0) ++(-45:.5) coordinate (X42);
\foreach \i in {1,...,4} { \foreach \j in {1,2} { \fill (X\i\j) circle
(2.8pt); } } \foreach \i in {1,...,4} { \draw (X\i1) -- (X\i2); }
\draw (X11) -- (X21) -- (X31) -- (X41) -- cycle; \draw (X12) -- (X22)
-- (X32) -- (X42) -- cycle; \draw (X12) -- (X21) -- (X32) -- (X41) --
cycle; \draw (X11) -- (X22) -- (X31) -- (X42) -- cycle;
\end{tikzpicture}
\caption{Splitting all vertices of $C_4$.}
\label{fig:splitC4}
\end{figure}
It can be seen that the binary model of $G$ is equivalent to that of
$C_4$ with $d=(4,4,4,4)$ in the sense that the ideals $I_{C_4}$ and
$I_{\gl(C_4)}$ are related to the ideals $I_G$ and $I_{\gl(G)}$ via a
renaming of the coordinates.
\qed
\end{example}

All examples where we could prove the interior point property involve
graphs where the toric ideal $I_{G}$ is generated in degree at most
four, and our proofs of the primary decomposition also depend on this
fact.
\begin{question}
If $I_{G}$ is generated in degree at most four, does this imply that
$(G,d)$ has the interior point property?
\end{question}

There are five graphs $G$ on $I_G$ such that $I_G$ (with $d_v=2$ for
all $v$) is not generated in degree four; and in this case $I_G$ is
generated in degree six~\cite{MBDB}.  Among these
graphs, $K_{3,3}$ and~$G_{154}$ are the only triangle-free graphs.  It
is a challenging problem to compute primary decompositions of
$I_{\gl(G)}$ for these two graphs.  By Example~\ref{ex:K33} $\Itt$ is
not radical.  The same method did not allow us to disprove radicality
of~$I_{\gl(G_{154})}$.  Note that $G_{154}$ can be obtained from
$K_{3,3}$ by deleting an edge.
\begin{figure}[ht]
\centering
\begin{tikzpicture}[scale=1.4]
\foreach \i in {1,...,3}
{
  \path (\i,0) coordinate (X\i);
  \path (\i,1) coordinate (Y\i);
  \fill (X\i) circle (2pt);
  \fill (Y\i) circle (2pt);
}
\foreach \i in {1,...,3}
{ \foreach \j in {1,...,3} 
  {
    \draw (X\i) -- (Y\j);
  }
}
\end{tikzpicture}
\hspace{1cm}
\begin{tikzpicture}[scale=1.4]
\foreach \i in {1,...,3}
{
  \path (\i,0) coordinate (X\i);
  \path (\i,1) coordinate (Y\i);
  \fill (X\i) circle (2pt);
  \fill (Y\i) circle (2pt);
}
\draw (X1) -- (X2) -- (X3) -- (Y3) -- (Y2) -- (Y1) -- (X1);
\draw (X1) -- (Y3);
\draw (X3) -- (Y1);
\end{tikzpicture}
\caption{$K_{3,3}$ and $G_{154}$}
\label{fig:sixVertexGraphs}
\end{figure}
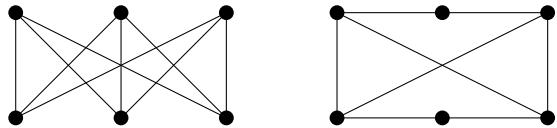

Theorem~\ref{thm:badexample} proves that for any two-connected
triangle-free graph, if the cardinalities $d=(d_v)_{v\in V}$ are
increased to all coincide with the same prime power, then this model does
not have the interior point property.  We conjecture that this
generalizes to many other graphs; i.e.~the situation should not
improve when the numbers $d_v$ are increased.  A similar phenomenon
occurs with Markov bases.  For instance, in the no-three-way
interaction model, the Markov basis becomes arbitrarily complicated as
two of the $d_v$ diverge~\cite{loera06:_markov_bases_of_three_way}.
\begin{conj}
Assume that $(G,d)$ does not have the positive margins (or interior
point) property.  If $d'_v\ge d_v$ for all $v\in V$, then $(G,d')$
does not have the positive margins (or interior point) property
either.
\end{conj}


\section{\texorpdfstring{Binary $N$-cycle models}{Binary N-cycle models}}
\label{sec:binary-n-cycles}

In this section we study the binary model of the $N$-cycle~$C_N$.  We
find a Markov basis (Theorem~\ref{t:ncycleMB}) and compute a prime
decomposition, showing that $I_{\gl(C_N)}$ is radical
(Theorem~\ref{t:ncycle-primedec}).  We then use this 
decomposition to prove the positive margins property
(Theorem~\ref{thm:Cn-pos-marg}).

We first describe a Markov basis of the toric ideal $I_{C_{N}}$.  A
Markov basis for this model was already presented in
\cite{develinsullivant03}.  Here, we construct a smaller Markov basis,
in order to simplify our proofs of the primary decomposition.
We use tableau notation to denote monomials and binomials in the
polynomial ring $\kk[p_x:x\in\cX]$.  The monomial
$p_{x_{1,1}x_{1,2}\dots x_{1,N}} \cdots p_{x_{t,t}x_{t,2}\dots
x_{t,N}}$ is represented by the following tableau with $t$ rows:
\begin{equation*}
\begin{bmatrix}
x_{1,1} & x_{1,2} & \dots & x_{1,N} \\
\vdots & \vdots & \vdots \\
x_{t,t} & x_{t,2} & \dots & x_{t,N}
\end{bmatrix}.
\end{equation*}
This notation greatly facilitates computations since applying moves to
a monomial merely corresponds to manipulating the entries of tableau
according to rules encoded by the moves.  Tableau calculations are
widely used in algebraic statistics, see for
example~\cite{drton09:_lectur_algeb_statis}.

All tableaux in this section are to be considered up to cyclic
symmetry, that is, a tableau represents also all other tableaux which
arise from making the $i$-th column the $(i+k)$-th column, where $k$ is
arbitrary and indices are considered modulo~$N$.

Denote $\cB_2$ the set of quadrics of the form
\begin{equation}
\label{e:ncycle-quads}
\begin{bmatrix}
a & A & b & B \\
a & A' & b & B' 
\end{bmatrix}-
\begin{bmatrix}
a & A' & b & B \\
a & A & b & B' 
\end{bmatrix},\quad 
\genfrac{}{}{0pt}{}{a,b \in [2], A,A' \in [2]^l, B,B'\in
[2]^{N-l-2},}{A\neq A', B\neq B',\quad 0 < l < N-2.}
\end{equation}
By convention, this means that the $a$- and $b$-columns are at
arbitrary non-adjacent positions in the binomial.  The quadrics in
$\cB_2$ therefore correspond to conditional independence statements of the form
$\ind{\{X_{k+1},\dots,X_{l-1}\}}{\{X_{l+1},\dots,X_{k-1}\}}[\{X_k,X_l\}]$
with $k$ and $l$ non-adjacent.   Note that $\gl(C_N)$ contains
further statements; but their quadrics are contained in the ideal
generated by those in~$\mathcal{B}_2$.

To each binary
state $K\in\cX$ there is a unique opposite state $\ol K\in\cX$,
defined by switching 1 and 2 in each component.
Let $\cB_4$ be the set of quartics of the form
\begin{equation}
\label{e:n-cycle-generic-quartics}
\begin{bmatrix}
A & B & C \\
A & \ol B & \ol C \\
\ol A & B & \ol C \\
\ol A & \ol B & C 
\end{bmatrix}
-
\begin{bmatrix}
A & B & \ol C \\
A & \ol B & C \\
\ol A & B & C \\
\ol A & \ol B & \ol C
\end{bmatrix},\quad 
\genfrac{}{}{0pt}{}{A \in [2]^k, B \in [2]^{l-k}, C\in [2]^{N-l},}{0 < k < l < N.}
\end{equation}

\begin{thm}
\label{t:ncycleMB}
For $N\geq 4$, the set $\cB_2 \cup \cB_4$ is a Markov basis of the
binary graphical model of the $N$-cycle.
\end{thm}
\begin{remark}\label{r:graphModel}
For $N=3$ the cycle is a complete graph, and therefore $I_G =
I_{\gl(G)} = 0$.  The single generator contained in $\cB_4$ defines
another interesting statistical model: the no-three-way interaction
model.  It is the hierarchical model of the graph $C_3$, considered as
a one-dimensional simplicial complex
(see~\cite{drton09:_lectur_algeb_statis}).  Such models were named
\emph{graph models} in~\cite{develinsullivant03}, in order to
distinguish them from graphical models.  For $N\geq 4$, all cliques of
the $N$-cycle are edges, and therefore the graphical model agrees with the
graph model.
\end{remark}
\begin{proof}[Proof of Theorem~\ref{t:ncycleMB}]
We prove that $\cB_2 \cup \cB_4$ is actually a Markov basis of the
graph model of $C_N$ for all $N\ge 1$.  We use induction on~$N$.  For
$N<3$ both $\cB_2$ and $\cB_4$ are empty, and the graph model contains
all probability distributions.  For $N=3$ the set $\cB_2$ is empty,
while $\cB_4$ contains only the defining quartic of the binary graph
model of~$C_3$. 

The $N$-cycle is a codimension-one toric fiber product of a chain of
length $(N-1)$ with a chain of length three (see~\cite{TFP-II}).
Since these chains are decomposable graphs, the Markov bases of these
chains consist of quadratic moves, corresponding to conditional
independence statements.  These Markov bases are slow-varying, in the
sense of~\cite{TFP-II}, and
gluing them yields moves in $\cB_2$.  By~\cite[Theorem~5.10]{TFP-II}, in
order to obtain a Markov basis of the $N$-cycle, we need to add
further quadrics (which belong to $\cB_2$) and a Markov basis of the
corresponding codimension-zero toric fiber product, which is the toric
fiber product of an $(N-1)$-cycle with the graph model of the
$3$-cycle.  By induction, we know the Markov bases of these smaller
cycles, and by~\cite[Theorem~5.4]{TFP-II} we need to lift these Markov
bases (and add some further quadrics that belong to $\cB_2$).  A lift
of a quadric gives again a quadric from $\cB_2$, and hence it suffices
to consider the quartics.

We first show that the ideal $I_{\cB_2\cup\cB_4}$ contains all tableaux of the form
\begin{equation}
\label{eq:3-lift}
\begin{bmatrix}
\1 & \1 & \1 & A \\
\1 & \2 & \2 & B \\
\2 & \1 & \2 & C \\
\2 & \2 & \1 & D
\end{bmatrix}
-
\begin{bmatrix}
\1 & \2 & \1 & A \\
\1 & \1 & \2 & B \\
\2 & \2 & \2 & C \\
\2 & \1 & \1 & D
\end{bmatrix},
\end{equation}
where each entry is a $\{1,2\}$-string of length at least one.
Suppose that there is a column~$k$ such that $A_k=C_k$.  Without loss
of generality, assume $A_k=C_k=1$.  Decompose the strings $A,B,C,D$
into substrings, such that $A = (A_lA_kA_r)$, and so on.  The tableau
calculation
\begin{multline*}
\begin{bmatrix}
\1 & \1 & \1 & A_l & 1 & A_r \\
\1 & \2 & \2 & B_l & b & B_r \\
\2 & \1 & \2 & C_l & 1 & C_r \\
\2 & \2 & \1 & D_l & d & D_r
\end{bmatrix}
\begin{matrix}
\ast \\ \\ \ast \\ \phantom{\ast}
\end{matrix}
\longrightarrow
\begin{bmatrix}
\1 & \1 & \2 & C_l & 1 & A_r \\
\1 & \2 & \2 & B_l & b & B_r \\
\2 & \1 & \1 & A_l & 1 & C_r \\
\2 & \2 & \1 & D_l & d & D_r
\end{bmatrix}
\begin{matrix}
\ast \\ \ast \\ + \\ +
\end{matrix}
\displaybreak[0]\\
\longrightarrow
\begin{bmatrix}
\1 & \2 & \2 & C_l & 1 & A_r \\
\1 & \1 & \2 & B_l & b & B_r \\
\2 & \2 & \1 & A_l & 1 & C_r \\
\2 & \1 & \1 & D_l & d & D_r
\end{bmatrix}
\begin{matrix}
\ast \\ \\ \ast \\ \phantom{+}
\end{matrix}
\longrightarrow
\begin{bmatrix}
\1 & \2 & \1 & A_l & 1 & A_r \\
\1 & \1 & \2 & B_l & b & B_r \\
\2 & \2 & \2 & C_l & 1 & C_r \\
\2 & \1 & \1 & D_l & d & D_r
\end{bmatrix}
\end{multline*}
shows that this move actually lies in the ideal generated by the
quadrics.  Here, $\ast$ and $+$ mark the rows to which a quadric has
been applied.  By symmetry, the same holds true if $B_k=D_k$ for
some~$k$.  If, in the tableau~\eqref{eq:3-lift}, $A_k\neq C_k$ and
$B_k\neq D_k$ for all $k$, then $C=\ol A$ and $D=\ol B$, and the move
is of the form
\begin{equation*}
\begin{bmatrix}
\1 & \1 & \1 & A \\
\1 & \2 & \2 & B \\
\2 & \1 & \2 & \ol A \\
\2 & \2 & \1 & \ol B
\end{bmatrix}
-
\begin{bmatrix}
\1 & \2 & \1 & A \\
\1 & \1 & \2 & B \\
\2 & \2 & \2 & \ol A \\
\2 & \1 & \1 & \ol B
\end{bmatrix}.
\end{equation*}
Hence, invoking the symmetry and exchanging $1\leftrightarrow2$ in
some columns of the last block, we may assume that any column in the
last block agrees with a column from either the first or the third
block.  If $
\begin{bmatrix}
A \\ B
\end{bmatrix}
=
\begin{bmatrix}
\1 & \1 \\
\2 & \1
\end{bmatrix}
$, then the lift belongs to~$\cB_4$.  Otherwise, using a rotation of
the cycle the move can be brought into the form
\begin{equation*}
\begin{bmatrix}
\1 & \1 & \1 & 1 &  A     & 1 \\
\1 & \2 & \2 & 1 &  B     & 2 \\
\2 & \1 & \2 & 2 &  \ol A & 2 \\
\2 & \2 & \1 & 2 &  \ol B & 1
\end{bmatrix}
-
\begin{bmatrix}
\1 & \2 & \1 & 1 & A     & 1 \\
\1 & \1 & \2 & 1 & B     & 2 \\
\2 & \2 & \2 & 2 & \ol A & 2 \\
\2 & \1 & \1 & 2 & \ol B & 1
\end{bmatrix}.
\end{equation*}
Applying quadrics to the first two rows transforms this into the move
\begin{equation*}
\begin{bmatrix}
\1 & \1 & \1 & 1 &  B     & 2 \\
\1 & \2 & \2 & 1 &  A     & 1 \\
\2 & \1 & \2 & 2 &  \ol A & 2 \\
\2 & \2 & \1 & 2 &  \ol B & 1
\end{bmatrix}
-
\begin{bmatrix}
\1 & \2 & \1 & 1 & B     & 2 \\
\1 & \1 & \2 & 1 & A     & 1 \\
\2 & \2 & \2 & 2 & \ol A & 2 \\
\2 & \1 & \1 & 2 & \ol B & 1
\end{bmatrix}.
\end{equation*}
In this move, the first and third entries of the last column agree,
and by the argument given above, it is a combination of quadrics.  Now
the theorem follows from the observation that, up to symmetry, any
lifted quartic is of the form~\eqref{eq:3-lift}.
\end{proof}

For any quartic $f$ of the form~\eqref{e:n-cycle-generic-quartics} let
\begin{equation*}
P_f = \< p_i: p_i\text{ divides neither } f^+\text{ nor }f^- \>.
\end{equation*}

\begin{lemma}
The ideals $P_f$ are prime ideals containing~$I_{\gl(C_{N})}$.
\end{lemma}
\begin{proof}
Clearly, $P_f$ is a monomial prime ideal.  Each generator of $I_{\gl(C_N)}$ is
of the form~\eqref{e:ncycle-quads}.
If the left term $\left[\begin{smallmatrix}
  a&A&b&B \\
  a&A'&b&B'
\end{smallmatrix}\right]$ is not contained in $P_f$, then it divides $f^+f^-$.
In this case either $A=A'$ or $B=B'$, and so both terms in
\eqref{e:ncycle-quads} agree.  Hence, $P_f$ contains all generators of $I_{\gl(C_N)}$.
\end{proof}

\begin{prop}
\label{prop:n-cycle-minprimes}
The minimal primes of $I_{\gl(C_{N})}$ are precisely the toric ideal 
$I_{C_{N}}$ and the monomial ideals~$P_{f}$.
\end{prop}

The proof of Proposition~\ref{prop:n-cycle-minprimes}
makes use of the following lemma.

\begin{lemma}
\label{l:toric-component-avoidance}
Let $f = f^+ - f^- \in \cB_4$ be a quartic generator of~$I_{C_{N}}$.  If the
variable $p_{i}$ divides neither  $f^+$ nor $f^-$, then $p_{i}f\in
I_{\gl(C_N)}$.
\end{lemma}
\begin{proof}
We have to show that $p_{i}f$ is a combination of quadrics coming from
conditional independence statements of the $N$-cycle.  Up to symmetry,
$p_if$ is of the form
\begin{equation*}
\begin{bmatrix}
\1 & \1 & \1\\
\1 & \2 & \2\\
\2 & \1 & \2\\
\2 & \2 & \1\\
K & L & M 
\end{bmatrix}
-
\begin{bmatrix}
\2 & \1 & \1\\
\1 & \2 & \1\\
\1 & \1 & \2\\
\2 & \2 & \2 \\
K & L & M 
\end{bmatrix},
\end{equation*}
where $p_i=p_{KLM}$.  We now transform $p_if$ into another binomial
$p_j\tilde f$ of total degree five using quadrics.  
Then $p_j\tilde f$ belongs to the toric ideal~$I_{C_N}$, and hence
$\tilde f\in I_{C_N}$.  Since $p_j\neq p_i$ the multidegree of $\tilde
f$ is not the multidegree of any quartic in~$\cB_4$.  Therefore,
$\tilde f$ must be a combination of quadrics, and we are done.

Using the symmetry we may assume that $K,L,M$ all contain at least
one~1, i.e.~$KLM=K_11K_2L_11L_2M_11M_2$.  
The tableau calculation
\begin{multline*}
\begin{bmatrix}
\2 & \1 & \1\\
\1 & \2 & \1\\
\1 & \1 & \2\\
\2 & \2 & \2 \\
K & L & M 
\end{bmatrix}
\begin{matrix}
\phantom{\ast} \\ \\ \ast \\ \\ \ast
\end{matrix}
\longrightarrow
\begin{bmatrix}
~~\2~2~\2~~ & ~~1~\1~1~~ & \1\\
~~\1~1~\1~~ & ~~2~\2~2~~ & \1\\
~~\1~1 K_2  &  L_1\1~1~~ & \2\\
~~\2~2~\2~~ & ~~2~\2~2~~ & \2 \\
 K_1 1~\1~~ & ~~1~\1 L_2 & M 
\end{bmatrix}
\begin{matrix}
\phantom{\ast} \\ \ast \\ \\ \\ \ast
\end{matrix}
\longrightarrow
\begin{bmatrix}
~~\2~2~\2~~ & ~~\1~1~\1~~ & ~~\1~1~\1~~\\
 K_1 1~\1~~ & ~~\2~2~\2~~ & ~~\1~1 M_2 \\
~~\1~1 K_2  &  L_1 1~\1~~ & ~~\2~2~\2~~\\
~~\2~2~\2~~ & ~~\2~2~\2~~ & ~~\2~2~\2~~ \\
~~\1~1~\1~~ & ~~\1~1 L_2  &  M_1 1~\1~~
\end{bmatrix}
\begin{matrix}
\ast \\ \\ \\ \\ \ast
\end{matrix}
\displaybreak[0]\\
\longrightarrow
\begin{bmatrix}
~~\2~2~\2~~ & ~~\1~1 L_2  &  M_1 1~\1~~\\
 K_1 1~\1~~ & ~~\2~2~\2~~ & ~~\1~1 M_2 \\
~~\1~1 K_2  &  L_1 1~\1~~ & ~~\2~2~\2~~\\
~~\2~2~\2~~ & ~~\2~2~\2~~ & ~~\2~2~\2~~ \\
~~\1~1~\1~~ & ~~\1~1~\1~~ & ~~\1~1~\1~~
\end{bmatrix}
\end{multline*}
shows how to transform the second term of $p_if$ such that the
resulting binomial is of the form $p_{\1\1\1}\tilde f$.
\end{proof}

\begin{proof}[Proof of Proposition~\ref{prop:n-cycle-minprimes}]
Let $p\in V(I_{\gl(C_{N})})\setminus V(I_{C_N})$.  Then there is a
quartic $f\in\cB_4$ such that $f(p)\neq 0$.
Lemma~\ref{l:toric-component-avoidance} implies that $p_K=0$ for all
$K$ such that $p_K$ does not divide $f^+f^-$, and hence $p\in V(P_f)$.
Clearly, the ideals $P_{f}$ are all distinct.  By symmetry, they are
all minimal primes.
\end{proof}

\begin{thm}
\label{t:ncycle-primedec}
The global Markov ideal $I_{\gl(C_{N})}$ is  radical and 
has prime decomposition
\[
I_{\gl(C_{N})} = 
 I_{C_N} \cap \bigcap_{f \in
\cB_4} P_f.
\]
\end{thm}
\begin{proof}
The intersection $J \defas I_{C_N} \cap \bigcap_{f \in \cB_4} P_f$ is
a binomial ideal, because, by
Proposition~\ref{prop:n-cycle-minprimes}, it is the radical of the
binomial
ideal~$I_{\gl(C_{N})}$~\cite[Theorem~3.1]{eisenbud96:_binom_ideal}.
Therefore, it suffices to consider an arbitrary binomial $p^u - p^v \in
J$ and show that it is contained in~$I_{\gl(C_{N})}$.  Since $J$ is
homogeneous in the multigrading of the toric ideal $I_{C_{N}}$, there
exists a sequence $v=u_0,u_1,\dots,u_r=u$ such that $u_i-u_{i-1}$ is a
move in the Markov basis $\cB_2\cup\cB_4$ of~$I_{C_{N}}$.  If only
quadratic moves are necessary, then $p^u-p^v\in I$.  Assume that
$u_i-u_{i-1}$ is the first quartic move, and let $f$ be the
corresponding quartic binomial.  Then $p^v-p^{u_{i-1}}\in
I_{\gl(C_N)}\subset P_f$, and hence $p^{u_{i-1}}-p^u\in P_f$.
Therefore, $p^{u_{i-1}}$ must be divisible by a variable
generating~$P_f$; and by definition, $p^{u_{i}}$ is divisible by the
same variable.  Hence, $p^{u_i}-p^{u_{i-1}}\in I_{\gl(C_N)}$ by
Lemma~\ref{l:toric-component-avoidance}.  Iteration of this argument
shows $p^u-p^v\in I_{\gl(C_N)}$.
\end{proof}

\begin{remark}
The minimal primes of $I_{\gl(C_N)}$ 
are exactly witnessed by degree four binomials in the Markov basis of
$I_{C_N}$.  More precisely, if $f \in \cB_4$ then $(I_{C_N}:f)$ is a
minimal prime, and all minimal primes arise in this way.
\end{remark}

\begin{proof}[Proof of Theorem~\ref{thm:Cn-pos-marg}]
Let $P=P_f$ be one of the minimal primes, where $f$ is a quartic of
the form~\eqref{e:n-cycle-generic-quartics}.  Then $m_P = f^+ f^-$.
Since $N\ge 4$, the quartic $f$ has at least two neighbouring columns
$i,i+1$ which are identical (up to symmetry).  Hence not all
components of the $\{i,i+1\}$-marginal of the exponent vector of $m_P$
can be positive.
\end{proof}


\section{\texorpdfstring{The complete bipartite graph $K_{2,N-2}$}{The complete bipartite graph K(2,N-2)}}
\label{sec:K2n}

In this section we study the complete bipartite graph $K_{2,N-2}$ with
vertex sets $\{1,2\}$, $\{3,\dots,N\}$ and with $d_1=d_2=2$ and
arbitrary $d_3,\dots,d_N$.  A Markov basis of the graphical model is
presented in Theorem~\ref{thm:Markov-basis-K2N}.
Using this Markov basis we compute a prime decomposition and show that
$I_{\gl(K_{2,N-2})}$ is radical (Theorem~\ref{thm:decomposition-K2N}).
With this decomposition we prove that for $N>4$ the complete
bipartite graph does not satisfy the positive margins property
(Theorem~\ref{thm:K22-pos-marg}), but the interior point property
(Theorem~\ref{thm:K2N-final-answer}).

The set $\gl(K_{2,N-2})$ consists of the CI statement
$\ind{X_1}{X_2}[\{X_3,\dots,X_N\}]$ and all statements
$\ind{X_A}{X_B}[\{X_1,X_2,X_C\}]$, where $A,B,C$ is a partition of
$\{3,\dots,N\}$.  The variables of the polynomial ring
$\kk[p_x:x\in\cX]$ can be arranged in a $(2\times 2\times
d_3\times\dots\times d_n)$-tensor ${p} = (p_{ijK} : i\in[2], j\in [2],
K\in\prod_{s=3}^N[d_s])$.  Define $(d_3\times\dots\times d_N)$-tensors
$A^{ij}$ and $(2\times 2)$-matrices $B^{K}$ via $A^{ij}_{K} =
B^{K}_{ij} := p_{ijK}$.  Then $A^{ij}$ and $B^{K}$ are slices of $p$.
The two sets of CI statements in $\gl(K_{2,N-2})$ correspond to the
two ideals
\begin{gather*}
I_{1}=\< 2 \times 2 \mbox{ minors of flattenings of } A^{ij}, i,j \in
\set{1,2}\> \quad \mbox{and} \\
I_{2} = \< 2 \times 2 \mbox{ minors of
} B^{K}, K\in\textstyle \prod_{i=3}^N[d_i]\>.
\end{gather*}
In $I_{1}$ we take all flattenings of the $(N-2)$-way tensor $A^{ij}$
down to a matrix and compute the $2 \times 2$ minors of those
matrices.  With this notation we have $I_{\gl(K_{2,N-2})} = I_{1} +
I_{2}$.  The quadratic generators of $I_2$
are of the form
\begin{subequations}
\label{eq:K2N-quads}
\begin{equation}
\label{eq:K2N-quads-1}
\begin{bmatrix}
1 & 1 & K \\
2 & 2 & K
\end{bmatrix}
-
\begin{bmatrix}
1 & 2 & K \\
2 & 1 & K
\end{bmatrix},
\end{equation}
and up to symmetry the generators of $I_1$ are of the form
\begin{equation}
\label{eq:K2N-quads-2}
\begin{bmatrix}
i & j & K & L \\
i & j & K'& L' \\
\end{bmatrix}-
\begin{bmatrix}
i & j & K'& L \\
i & j & K & L' \\
\end{bmatrix},
\end{equation}
\end{subequations}
where $i,j,K,K',L,L'$ are arbitrary in their respective domains (here,
the symmetry says that we can permute the last $N-2$ columns).
\begin{thm}
\label{thm:Markov-basis-K2N}
A Markov basis of the toric ideal $I_{K_{2,N-2}}$ with $d_{1} = d_{2}
= 2$ consists of the quadratic generators~\eqref{eq:K2N-quads} of
$I_{\gl(K_{2,N-2})}$ and the quartic binomials
\begin{equation*}
B_{3;k_1,k_2}^{L_{11}L_{12}L_{21}L_{22}} :=
\begin{bmatrix}
1 & 1 & k_1 & L_{11} \\
1 & 2 & k_2 & L_{12} \\
2 & 1 & k_2 & L_{21} \\
2 & 2 & k_1 & L_{22}
\end{bmatrix}
-
\begin{bmatrix}
1 & 1 & k_2 & L_{11} \\
1 & 2 & k_1 & L_{12} \\
2 & 1 & k_1 & L_{21} \\
2 & 2 & k_2 & L_{22}
\end{bmatrix}
\end{equation*}
for all $k_1,k_2\in[d_3]$ and
$L_{11},L_{12},L_{21},L_{22}\in[d_4]\times\dots\times[d_n]$, and the
corresponding quartics $B_{a;k_1,k_2}^{L_{11}L_{12}L_{21}L_{22}}$ for
$a=4,\dots,N$, where the roles of the columns 3 and $a$ are exchanged
in the above equation.
\end{thm}

\begin{proof}
The proof is by induction on $N$.  The base case $N = 4$
is~\cite[Cor.~2.2]{TFP-II}.  Suppose that
Theorem~\ref{thm:Markov-basis-K2N} holds for some $N$.  We show that
it also holds for $N+1$.  The graph $K_{2,N-1}$ is obtained by gluing
the graph $K_{2,N-2}$ and the graph $K_{2,1}$ at the first two
vertices.  This is a codimension-one toric fiber product, which is
slow-varying, since all quartic generators
$B_{a;k_1,k_2}^{L_{11}L_{12}L_{21}L_{22}}$ project to the zero
polynomial when just considering their indices associated to the first
two vertices, see~\cite[\S 5.3]{TFP-II}.

We first show that the set $\tilde\cB$ which consists of all quartics
of the form $B_{a;k_1,k_2}^{L_{11}L_{12}L_{21}L_{22}}$ and the
quadratic moves of the form~\eqref{eq:K2N-quads-2} is a Markov basis
of the associated codimension-zero toric fiber product, which is the
graph model of the graph $\tilde K_{2,N-1}$ with vertex set $[N+1]$
and edge set $\left\{(i,j) : i<j\le N+1, i\le 2\right\}$.  Again, this
can be proved by induction: The induction base $\tilde K_{2,1}=C_3$ is
discussed in Remark~\ref{r:graphModel}.
By~\cite[Theorem~5.4]{TFP-II}, a Markov basis of $I_{\tilde
K_{2,N-1}}$ consists of the quadrics generating $I_2$ and lifts of
elements of the Markov bases of $I_{\tilde K_{2,N-2}}$ and $I_{\tilde
K_{2,1}}$.  The lift of a quadratic generator of $I_{\tilde
K_{2,N-1}}$ is a quadratic generator of $I_2$.  The lift of a quartic
generator of $I_{\tilde K_{2,N-1}}$ or $I_{\tilde K_{2,1}}$ is of the
form $B_{a;k_1,k_2}^{L_{11}L_{12}L_{21}L_{22}}$.  This proves that
$\tilde\cB$ is a Markov basis of $I_{\tilde K_{2,N-1}}$.

By~\cite[Theorem~5.10]{TFP-II}, we can obtain a Markov basis of
$I_{K_{2,N-1}}$ from $\tilde\cB$ by adding additional quadrics of the
form~\eqref{eq:K2N-quads-2} and moves obtained by gluing elements from
the Markov bases of $I_{K_{2,1}}$ and $I_{K_{2,N-2}}$.  Since
$K_{2,1}$ is decomposable, the quadratic moves of the
form~\eqref{eq:K2N-quads-1} alone form a Markov basis of $I_{2,1}$ (no
quartics are needed).  These quadratic moves can only be glued with
the corresponding quadratic generators from $I_{K_{2,N-1}}$, and this
gluing procedure yields all quadratic moves of the
form~\eqref{eq:K2N-quads-1}.

To sum up, the quartic moves and the quadratic moves of the
form~\eqref{eq:K2N-quads-2} belong to the associated codimension-one
toric fiber product, and the quadratic moves of the
form~\eqref{eq:K2N-quads-1} arise iteratively from the quadratic
generators of $I_{K_{2,1}}$.
\end{proof}

Now we proceed to describe the other minimal primes of the
ideal~$I_{\gl(K_{2,N-2})}$.
\begin{lemma}
\label{lem:associated-primes}
Let $a,b\in\{3,\dots,N\}$, and let $C\subset[d_a]$ and $D\subset[d_b]$.
Then the ideal $P_{a,C,b,D}$ generated by $I_{\gl(K_{2,N-2})}$ and the variables
\begin{equation*}
\{p_{11K} : K_a\in C\} \cup \{p_{12K} : K_b\in D\}
\cup \{p_{21K} : K_b \notin D\} \cup \{p_{22K} : K_a\notin C\}.
\end{equation*}
is a prime ideal containing~$I_{\gl(K_{2,N-2})}$.
\end{lemma}

\begin{proof}
$P_{a,C,b,D}$ is prime since it is a sum of geometrically
prime ideals which are defined in disjoint sets of
variables.  This can be seen as follows: First, the variables in
$P_{a,C,b,D}$ generate a monomial prime ideal.  Second, all binomial
generators of $I_2$ are redundant modulo that ideal, i.e.~they are
implied by the variables in~$P_{a,C,b,D}$.  Third, let $f = p^u- p^v$
be a binomial generator of~$I_1$.  Then $p^u$ contains a variable
generating $P_{a,C,b,D}$ if and only if $p^v$ contains a variable
in~$P_{a,C,b,D}$.  The binomials in $I_1$ which are not implied by the
variables in $P_{a,C,b,D}$ correspond to rank conditions on disjoint
slices of the tensor $p$; hence they generate a binomial prime ideal
over any field.
\end{proof}

\begin{prop}
  \label{prop:minimal-primes-K2N}
  All minimal primes of $I_{\gl(K_{2,N-2})}$ except the toric
  component $I_{K_{2,N-2}}$ are of the form~$P_{a,C,b,D}$.
  Specifically:
  \begin{enumerate}
  \item If $N=4$, then the set of minimal primes consists of the toric
  component and all primes of the form $P_{a,C,a,D}$, where $a \in
  \{3, 4\}$, $\emptyset\neq C\neq[d_a]$, and $\emptyset\neq
  D\neq[d_a]$.
  \item If $N>4$, then the set of minimal primes consists of the toric
  component and all primes of the form $P_{a,C,b,D}$, where $a,b \in
  \{3, \ldots, N \}$, $\emptyset\neq C\neq[d_a]$, and $\emptyset\neq
  D\neq[d_b]$.
\end{enumerate}
\end{prop}

The proof of Proposition~\ref{prop:minimal-primes-K2N} 
makes use of the following lemma.
\begin{lemma}\label{lemma:johannes}
For any $K\in\prod_{i=3}^n[d_i]$, 
\[I_{\gl(K_{2,N-2})} : p_{11K}p_{22K} = I_{\gl(K_{2,N-2})} : p_{12K}p_{21K} = 
I_{K_{2,N-2}}.\]
In particular, if $P$ is a minimal prime of $I_{\gl(K_{2,N-2})}$ and
not the toric component $I_{K_{2,N-2}}$, then $p_{11K}p_{22K} \in P$
and $p_{12K}p_{21K} \in P$.
\end{lemma}

\begin{proof}
We need to prove that both
$p_{11kL}p_{22kL}B_{a;i_1i_2}^{L_{11}L_{12}L_{21}L_{22}}$ and
$p_{12kL}p_{21kL}B_{a;i_1i_2}^{L_{11}L_{12}L_{21}L_{22}}$ belong to
$I_{\gl(K_{2,N-2})}$, and by symmetry it suffices to treat the first
binomial.  Moreover, by symmetry we may assume $a=3$.
The calculation can be done using tableau notation:
\begin{multline*}
\begin{bmatrix}
1&1&k_1&L_{11}\\
1&2&k_2&L_{12}\\
2&1&k_2&L_{21}\\
2&2&k_1&L_{22}\\
1&1&k&L\\
2&2&k&L
\end{bmatrix}
\begin{matrix}
\ast \\ + \\ \\ \\ \ast \\ +
\end{matrix}
\longrightarrow
\begin{bmatrix}
1&1& k &L_{11}\\
1&2&k_2&L_{12}\\
2&1&k_2&L_{21}\\
2&2& k &L_{22}\\
1&1&k_1&L\\
2&2&k_1&L
\end{bmatrix}
\begin{matrix}
\phantom{\ast}\\ \\ \\ \\ \ast \\ \ast
\end{matrix}
\longrightarrow
\begin{bmatrix}
1&1& k &L_{11}\\
1&2&k_2&L_{12}\\
2&1&k_2&L_{21}\\
2&2& k &L_{22}\\
1&2&k_1&L\\
2&1&k_1&L
\end{bmatrix}
\begin{matrix}
\phantom{\ast}\\ \ast \\ + \\ \\ \ast \\ +
\end{matrix}
\displaybreak[0]\\
\longrightarrow
\begin{bmatrix}
1&1& k &L_{11}\\
1&2&k_1&L_{12}\\
2&1&k_1&L_{21}\\
2&2& k &L_{22}\\
1&2&k_2&L\\
2&1&k_2&L
\end{bmatrix}
\begin{matrix}
\phantom{\ast}\\ \\ \\ \\ \ast \\ \ast
\end{matrix}
\longrightarrow
\begin{bmatrix}
1&1& k &L_{11}\\
1&2&k_1&L_{12}\\
2&1&k_1&L_{21}\\
2&2& k &L_{22}\\
1&1&k_2&L\\
2&2&k_2&L
\end{bmatrix}
\begin{matrix}
\ast \\ \\ \\ + \\ \ast \\ +
\end{matrix}
\longrightarrow
\begin{bmatrix}
1&1&k_2&L_{11}\\
1&2&k_1&L_{12}\\
2&1&k_1&L_{21}\\
2&2&k_2&L_{22}\\
1&1& k &L\\
2&2& k &L
\end{bmatrix}.
\end{multline*}
Here, the first tableau and the last tableau correspond to the two monomials
of $p_{11kL}p_{22kL}B_{3;k_1k_2}^{L_{11}L_{12}L_{21}L_{22}}$.
\end{proof}

\begin{proof}[Proof of Proposition~\ref{prop:minimal-primes-K2N}]
We use a set-theoretic argument.  Let
$p$ 
be any point in the variety of~$I_{\gl(K_{2,N-2})}$, and consider the
$(d_3\times\dots\times d_N)$-tensors $A^{ij}$ with $A^{ij}_{K} =
p_{ijK}$.  If no coordinate of $p$ vanishes, then $p$ is contained in
the variety of~$I_{K_{2,N-2}}$.  Therefore, suppose $p_{ijK} = 0$ for
some $ijK\in\prod_{i=1}^N[d_i]$.
The CI statements $\gl(K_{2,N-2})$ imply that all $A^{ij}$ have rank
one.
Hence there must be an index $a$ such that $p_{ijK'}=0$ whenever
$K'\in\prod_{i=3}^N[d_i]$ satisfies $K'_a=K_a$.
In other words, for all $i,j\in\{1,2\}$, the pattern of zeros within
$A^{ij}$ is a union of $(N-3)$-dimensional slices.

For each $a=3,\dots,N$ let $E_a^{ij}$ be the largest subset of $[d_a]$ such
that $p_{ijK}=0$ whenever $K_a\in E_a^{ij}$.  Then $A^{ij}_K\neq 0$ if and only if
$K\in([d_3] \setminus E_3^{ij}) \times\dots\times ([d_N] \setminus E_N^{ij})$.  By
Lemma~\ref{lemma:johannes}, if 
$p$ does not lie in the toric component, then 
$p_{\ol{ij}K}=A^{\ol{ij}}_K=0$ for all $K\in([d_3] \setminus E_3^{ij})
\times\dots\times ([d_N] \setminus E_N^{ij})$ (remember that $\ol{ij}$
denotes the ``opposite'' string to $ij$, obtained by exchanging
$0\leftrightarrow1$ in each position).  Again, each of these entries
must be contained in an $(N-3)$-slice of zeros.  Hence there must be
an index $a_{ij}$ such that $[d_{a_{ij}}]\setminus E_{a_{ij}}^{ij}$ is
a subset of $E_{a_{ij}}^{\ol i\ol j}$; for otherwise, if for each $a$
there exists $i_a\in([d_a]\setminus E_a^{\ol i\ol j})\cap
([d_a]\setminus E_a^{ij})$, then $p_{ijI} p_{\ol i\ol jI}\neq 0$,
where $I_a = i_a$ for all~$a$.  This implies that we can
find 
subsets $C\subseteq[d_{a_{11}}]$, $D\subseteq[d_{a_{12}}]$, such that
$p\in V(P_{a_{11},C,a_{12},D})$.  This shows the first statement, and
it remains to see that certain choices of $P_{a,C,b,D}$ do not appear.

If $C=\emptyset$, then $P_{a,C,b,D}$ contains the toric component:
Indeed, $P_{a,\emptyset,b,D}$ contains all monomials of the form
$p_{22K}$, and hence $P_{a,\emptyset,b,D}$ contains all quartics.
Therefore, $P_{a,\emptyset,b,D}$ is not a minimal prime, and the same
is true if $C=[d_a]$, $D=\emptyset$ or $D=[d_b]$.  Similarly, if
$N=4$, then $P_{3,C,4,D}$ contains
both monomials of any quartic.

It follows that all minimal primes are among the ideals $P_{a,C,b,D}$
listed in the statement of the theorem.  It remains to show that all
these ideals are indeed minimal primes.  Note that each of these
ideals 
contains a different set of variables of
the same size, so they do not contain each other.  Furthermore, they
each leave out at least one of the quartic moves. 
Indeed, choose an index $c\in\{3,\dots,N\}\setminus\{a,b\}$, 
choose $k_1,k_2\in[d_c]$, and choose
$L_{11},L_{12},L_{21},L_{22}\in\prod_{i\ge 3, i\neq c}[d_i]$ such
that $(L_{11})_a\notin C, (L_{22})_a\in C, (L_{12})_b \notin D,
(L_{21})_b \in D$.  Then $P_{a,C,b,D}$ does not contain 
$B_{a;k_1k_2}^{L_{11}L_{12}L_{21}L_{22}}$, and so $I_{K_{2,N-2}}\not\subseteq P_{a,C,b,D}$.
\end{proof}

\begin{thm}
\label{thm:decomposition-K2N}
The global Markov ideal $I_{\gl(K_{2,N-2})}$ is a radical ideal
when $d_{1} = d_{2} = 2$, with irredundant prime decomposition
\[
I_{\gl(K_{2,N-2})}  =  I_{K_{2,N-2}} \cap  \bigcap P_{a,C,b,D}
\]
with the intersection running over all $a,b \in \{3, \ldots, N\}$,
$ C \subset [d_{a}]$, $D \subset [d_{b}]$, $C,D \neq \emptyset$.
When $N = 4$, we also require $a = b$.
\end{thm}
\begin{proof}
Let $J$ be the intersection of the toric component with all minimal
primes~$P_{a,C,b,D}$.  
Lemma~\ref{lem:associated-primes} shows $I\subseteq J$, and it remains
to show the opposite inclusion.
It suffices to consider binomials: By
Proposition~\ref{prop:minimal-primes-K2N} the radical of $I$ equals $J$,
and therefore $J$ is generated by
binomials~\cite[Theorem~3.1]{eisenbud96:_binom_ideal}.

Let $p^u-p^v\in J$.  If there exists a prime $P_{a,C,b,D}$ such that
$p^u$ does not contain any of the variables defining $P_{a,C,b,D}$, then
$p^u-p^v$ actually belongs to the ideal generated by the binomial part
of $P_{a,C,b,D}$, and hence $p^u-p^v\in I$.  Therefore, we may assume in
the following that for any prime $P_{a,C,b,D}$ the monomial $p^u$
contains at least one of the variables defining~$P_{a,C,b,D}$.
Since 
$p^u-p^v\in I_{K_{2,N-2}}$ there is a decomposition
$p^u - p^v = \sum_{i=1}^r (p^{u_{i-1}} - p^{u_i})$,
where $u_0=u$, $u_r=v$, and $u_{i-1}-u_i$ is an element of the Markov
basis.  If $u_0-u_1$ is a quadratic element of the Markov basis, then
$p^u-p^v-p^{u}+p^{u_1}$ is an element of $J$ and belongs to $I$ if and
only if $p^u-p^v$ belongs to $I$ (since $p^u-p^{u_1}$ is contained in
$I$ as well as in each minimal prime).

Assume that $u_0-u_1$ corresponds to a quartic move, say
$B_{3;k_1k_2}^{L_{11}L_{12}L_{21}L_{22}}$ for some $k_1,k_2\in[c]$ and
$L_{11},L_{12},L_{21},L_{22}\in[d]$.
We use induction on the number of mismatches of $L_{11}$
and $L_{22}$ and the number of mismatches of $L_{12}$ and $L_{21}$ to show that
we can replace this quartic with a combination of quadratic Markov
moves.  This shows that $p^u-p^{u_1}$ actually lies in~$I$.
If $L_{11}=L_{22}$, then the calculation
\begin{multline*}
\begin{bmatrix}
1&1&k_1&L_{11}\\
1&2&k_2&L_{12}\\
2&1&k_2&L_{21}\\
2&2&k_1&L_{11}
\end{bmatrix}
\begin{matrix}
\ast \\ \\ \\ \ast
\end{matrix}
\longrightarrow
\begin{bmatrix}
1&2&k_1&L_{11}\\
1&2&k_2&L_{12}\\
2&1&k_2&L_{21}\\
2&1&k_1&L_{11}
\end{bmatrix}
\begin{matrix}
\ast \\ \ast \\ + \\ +
\end{matrix}
\displaybreak[0]\\
\longrightarrow
\begin{bmatrix}
1&2&k_2&L_{11}\\
1&2&k_1&L_{12}\\
2&1&k_1&L_{21}\\
2&1&k_2&L_{11}
\end{bmatrix}
\begin{matrix}
\ast \\ \\ \\ {\ast}
\end{matrix}
\longrightarrow
\begin{bmatrix}
1&1&k_2&L_{11}\\
1&2&k_1&L_{12}\\
2&1&k_1&L_{21}\\
2&2&k_2&L_{11}
\end{bmatrix}
\end{multline*}
shows that $B_{3;k_1k_2}^{L_{11}L_{12}L_{21}L_{11}}$ is a combination of quadratic
Markov moves and hence lies in~$I$.  By symmetry, the same is true
when $L_{12}=L_{21}$.  Therefore, we may assume $L_{11}\neq L_{22}$ and $L_{12}\neq
L_{21}$ in the following.

As shown above, there exists a variable $p_{ijkL}$
that divides $p^u$ and $p^{u_1}$ (in particular, $p_{ijkL}$ is not
involved in $B_{3;k_1k_2}^{L_{11}L_{12}L_{21}L_{22}}$).  Without loss of
generality assume $i=j=2$.  If there exists $a\ge 2$ such that
$L_a=(L_{11})_a\neq(L_{22})_a$, then we can apply the moves
\begin{gather*}
\begin{bmatrix}
2 & 2 & k_1 & L_{22} \\
2 & 2 & k & L
\end{bmatrix}
-
\begin{bmatrix}
2 & 2 & k_1 & L_{22}' \\
2 & 2 & k & L'
\end{bmatrix},
\text{ and }
\begin{bmatrix}
2 & 2 & k_2 & L_{22} \\
2 & 2 & k & L
\end{bmatrix}
-
\begin{bmatrix}
2 & 2 & k_2 & L_{22}' \\
2 & 2 & k & L'
\end{bmatrix},
\\ \text{with }
(L_{22}')_b=
\begin{cases}
(L_{22})_b & \text{ if }b\neq a,\\
L_a & \text{ if }b = a,\\
\end{cases}
\quad \text{and} \quad 
L'_b=
\begin{cases}
L_b & \text{ if }b\neq a,\\
(L_{22})_a & \text{ if }b = a,\\
\end{cases}
\end{gather*}
to $p^u$ and~$p^{u_1}$.  This effectively replaces $L_{22}$ by $L_{22}'$,
and $L_{11}$ and $L_{22}'$ agree in more components than $L_{11}$ and $L_{22}$.

By symmetry, if there exist $i,j\in\{0,1\}$ and $a\ge2$ with
$L=L_{ij})_a\neq (L_{\ol{ij}})_a$, then we can apply quadratic moves
to make $L_{ij}$ and $L_{\ol{ij}}$ more similar to each other.  Now we
may assume that each variable $p_{ijkL}$ that divides $p^u$ satisfies
$L_a \neq (L_{\ol i\ol j})_a$.  We show that it is still possible
for some $i,j$ to reduce the number of mismatches between $L_{ij}$ and
$L_{\ol{ij}}$.

Choose indices $a$ and $b$ such that $(L_{11})_a\neq(L_{22})_a$ and
$(L_{12})_b\neq(L_{21})_b$.
We claim that in this case, there exist
$k_3,k_4\in[d_3]$ and $L_5,L_6\in\prod_{i=4}^n[d_i]$ such that
\begin{itemize}
\item either $(L_5)_a=(L_6)_a$ and $p_{11k_3L_5}p_{22k_4L_6}$ divides
$p^u$,
\item or $(L_5)_b=(L_6)_b$ and $p_{12k_3L_5}p_{21k_4L_6}$
divides~$p^u$.
\end{itemize}
Otherwise, $p^u$ would contain no defining variable of the prime
$P_{a,C,b,D}$ with
\begin{align*}
C&=\{l\in[d_a]: p_{11kL}\text{ does not divide }p^u\text{ for all }k\in[d_1], L\in\prod_{s=2}^n[d_s]\text{ with }L_a=l\}, \\
D&=\{l\in[d_b]: p_{12kL}\text{ does not divide }p^u\text{ for all }k\in[d_1], L\in\prod_{s=2}^n[d_s]\text{ with }L_b=l\}
\end{align*}
(note that if $N=4$, then $a=b$).  By symmetry it suffices to consider
the first case, i.e.~$(L_5)_a=(L_6)_a$ and $p_{11k_3L_5}p_{22k_4L_6}$
divides~$p^u$.
We can then apply the moves
\begin{gather*}
\begin{bmatrix}
1&1&k_1&L_{11}\\
1&1&k_3&L_5
\end{bmatrix}
-
\begin{bmatrix}
1&1&k_1&L_{11}'\\
1&1&k_3&L_5'
\end{bmatrix}
\text{ and }
\begin{bmatrix}
2&2&k_1&L_{22}\\
2&2&k_3&L_6
\end{bmatrix}
-
\begin{bmatrix}
2&2&k_1&L_{22}'\\
2&2&k_3&L_6'
\end{bmatrix}
\displaybreak[0]\\
\text{with }
(L_{11}')_c=
\begin{cases}
(L_{11})_c & \text{ if }c\neq a,\\
(L_5)_a & \text{ if }c = a,
\end{cases}
(L_5)'_c=
\begin{cases}
(L_5)_c & \text{ if }c\neq a,\\
(L_{22})_a & \text{ if }c = a,
\end{cases}
\displaybreak[0]\\
\text{and }
(L_{22}')_c=
\begin{cases}
(L_{22})_c & \text{ if }c\neq a,\\
(L_6)_a & \text{ if }c = a,
\end{cases}
(L_6)'_c=
\begin{cases}
(L_6)_c & \text{ if }c\neq a,\\
(L_{22})_a & \text{ if }c = a,\\
\end{cases}
\end{gather*}
to $p^u$ and $p^{u_1}$.  This effectively replaces $L_{11}$ by
$L_{11}'$ and $L_{22}$ by $L_{22}'$, and $L_{11}'$ and $L_{22}'$ agree
in more components than $L_{11}$ and $L_{22}$.  This proves the induction step and
shows that $p^u - p^v$ lies in~$I$.
\end{proof}

With this primary decomposition, we can analyze the 
positive margins property.

\begin{thm}
\label{thm:K22-pos-marg}
For $N\ge 4$, the complete bipartite graph $K_{2,N-2}$, where the
first group of nodes is binary, has the positive margins property if
and only if $N=4$.
\end{thm}
\begin{proof}
We check the condition in
Lemma~\ref{lem:graphical-positive-margins}.  If $N=4$, then each
minimal prime is of the form $P=P_{a,C,a,D}$.  Because of the symmetry
we may assume $a=3$.  Then
\begin{equation*}
m_P = \prod_{k,l : k\notin C} p_{11kl} \prod_{k,l : k\notin D} p_{12kl} 
\prod_{k,l : k\in D} p_{21kl} \prod_{k,l : k\in C} p_{22kl}.
\end{equation*}
Suppose that $C$ and $D$ intersect.  Then the $\{1,3\}$-marginal of
the exponent vector of $m_P$ is not strictly positive, since any
variable $p_{1jkl}$ that divides $m_P$ satisfies $k\notin C\cap D$.
Similarly, if $C$ and $D$ do not intersect, then $C$ intersects the
complement of $D$, and hence the $\{2,3\}$-marginal cannot be strictly
positive.

If $N>4$, then consider a prime of the form $P=P_{3,C,4,D}$.  Then
\begin{equation*}
m_P = \prod_{K : K_a\notin C} p_{11K} \prod_{K : K_b\notin D} p_{12K} 
\prod_{K : K_b\in D} p_{21K} \prod_{K : K_a\in C} p_{22K},
\end{equation*}
and the exponent vector has strictly positive margins: Indeed, take
for example the $\{1,a\}$-marginal.  For any $k\in[d_a]$, choose
$K,L\in[d_3]\times\dots\times[d_N]$ such that $K_b\notin D, L_b\in D$
and $K_a=k=L_a$.  Then $p_{12K}p_{21L}$ divides $m_P$, and hence the
$\{1,k\}$-count and the $\{2,k\}$-count of the $\{1,a\}$-marginal are
larger than zero.
\end{proof}

Finally, we want to prove that $K_{2,N-2}$ satisfies the interior
point property.  We first describe additional inequalities of the
marginal cone.
\begin{lemma}
\label{lem:K2n-cycle-inequalities}
Let $N\ge 4$, and assume $d_1=d_2=2$.  For any table $u\in\NN^n$,
denote $y = Au$ the vector of $K_{2,N-2}$-marginals, which has components $y^{ij}_{kl}$ for
$i=1,2$, $j=3,\dots,N$, and $(k,l)\in[2]\times [d_j]$.  Let $a,b$ such
that $3\le a < b\le N$ and let $C\subset [d_a]$, $D\subset[d_b]$ be
non-empty subsets such that $C\neq [d_a]$ and $D\neq[d_b]$.  For any
choice of $a,b,C,D$
\begin{equation}
\label{eq:K2n-facets}
\sum_{k\in C} y_{1k}^{1a} + \sum_{k\notin C}y_{2k}^{2a}
+ \sum_{l\in D} y_{1l}^{2b} - \sum_{l\in D} y_{1l}^{1b} \ge 0.
\end{equation}
\end{lemma}
\begin{proof}
It suffices to show that each unit vector in $\NN^n$
satisfies~\eqref{eq:K2n-facets}.  Consider the unit vector $e_x$
corresponding to $x\in\cX$.  If the last summand $\sum_{l\in D} y_{1l}^{1b}(e_x)$
vanishes, then~\eqref{eq:K2n-facets} holds.  Otherwise, $x_1=1$ and
$x_l\in D$, and so this sum equals one.  In this case, at least one of the
following three possibilities happens: either $x_a\in C$, or $x_a\notin C$ and $x_2=2$, or
$x_2=1$.  In any case, $\sum_{k\in C} y_{1k}^{1a} + \sum_{k\notin
C}y_{2k}^{2a} + \sum_{l\in D} y_{1l}^{2b}\ge 1$, and
so~\eqref{eq:K2n-facets} holds.
\end{proof}

\begin{thm}
\label{thm:K2N-final-answer}
Assume that $d_1=d_2=2$.  If $u\in\NN^n$ has strictly positive
$K_{2,N-2}$-margins and if $u$ satisfies all inequalities of the
form~\eqref{eq:K2n-facets} with strict inequality, then the fiber of
$u$ is connected by quadratic moves.
\end{thm}
\begin{proof}
We apply Lemma~\ref{lem:positive-margins}.  Let $P=P_{a,C,b,D}$ be a
minimal prime.  If $a=b$, then the proof of
Theorem~\ref{thm:K22-pos-marg} shows that the exponent vector $u_P$ of
$m_P$ has at least one vanishing marginal.
If $a\neq b$, then a direct verification shows that $u_P$ satisfies
\begin{equation*}
\sum_{k\in C} y_{1k}^{1a} + \sum_{k\notin C}y_{2k}^{2a}
+ \sum_{l\in D} y_{1l}^{2b} - \sum_{l\in D} y_{1l}^{1b} = 0.\qedhere
\end{equation*}
\end{proof}

\begin{proof}[Proof of Theorem~\ref{thm:K22-pos-marg-int-point}] 
Combine Theorems~\ref{thm:K22-pos-marg} and~\ref{thm:K2N-final-answer}.
\end{proof}

\bibliographystyle{amsplain}
\bibliography{math}

\providecommand{\bysame}{\leavevmode\hbox to3em{\hrulefill}\thinspace}
\providecommand{\MR}{\relax\ifhmode\unskip\space\fi MR }
\providecommand{\MRhref}[2]{%
  \href{http://www.ams.org/mathscinet-getitem?mr=#1}{#2}
}
\providecommand{\href}[2]{#2}
\begin{thebibliography}{10}

\bibitem{4ti2}
{\relax 4ti2}~{\relax }team, \emph{4ti2---a software package for algebraic,
  geometric and combinatorial problems on linear spaces}, {a}vailable at
  www.4ti2.de, 2007.

\bibitem{Besag74}
J.~Besag, \emph{Spatial interaction and the statistical analysis of lattice
  systems}, Journal of the Royal Statistical Society, Series B \textbf{36}
  (1974), 192--236.

\bibitem{birkner09:_polyh}
Ren\'{e} Birkner, \emph{Polyhedra: a package for computations with convex
  polyhedral objects}, J. of Software for Algebra and Geometry \textbf{1}
  (2009), 11--15.

\bibitem{Bourbaki50:Algebre_4-7}
Nicolas Bourbaki, \emph{{\'E}l{\'e}ments de math{\'e}matique: Alg{\`e}bre ;
  chapitre 4--7}, no.~2, Hermann, 1950.

\bibitem{bunea00:_mcmc_i_j_k}
Florentina Bunea and Julian Besag, \emph{{MCMC} in $i \times j \times k$
  contingency tables}, Monte Carlo Methods (Neal Madras, ed.), Fields Institute
  Communications, vol.~26, AMS and Fields Institute, 2000.

\bibitem{chen10:_markov}
Yuguo Chen, Ian Dinwoodie, and Ruriko Yoshida, \emph{{M}arkov chains, quotient
  ideals, and connectivity with positive margins}, Algebraic and Geometric
  Methods in Statistics (Pablo Gibilisco, Eva Riccomagno, Maria~Piera Rogantin,
  and Henry~P. Wynn, eds.), Cambridge University Press, Cambridge, UK, 2010,
  pp.~99--110.

\bibitem{yuguo06:_sequen}
Yuguo Chen, Ian~H. Dinwoodie, and Seth Sullivant, \emph{Sequential importance
  sampling for multiway tables}, Ann. Statist (2006), 523--545.

\bibitem{DenesKeedwell74:Latin_squares_and_applications}
József D{\'e}nes and A.D. Keedwell, \emph{Latin squares and their
  applications}, Academic Press, 1974.

\bibitem{develinsullivant03}
Mike Develin and Seth Sullivant, \emph{{M}arkov bases of binary graph models},
  Annals of Combinatorics \textbf{7} (2003), 441--466.

\bibitem{diaconis98:_lattic}
Persi Diaconis, David Eisenbud, and Bernd Sturmfels, \emph{Lattice walks and
  primary decomposition}, Mathematical Essays in Honor of Gian-Carlo Rota
  (B.~Sagan and R.~Stanley, eds.), Progress in Mathematics, vol. 161,
  Birkhauser, Boston, 1998, pp.~173--193.

\bibitem{diaconissturmfels98}
Persi Diaconis and Bernd Sturmfels, \emph{Algebraic algorithms for sampling
  from conditional distributions}, Annals of Statistics \textbf{26} (1998),
  363--397.

\bibitem{drton09:_lectur_algeb_statis}
Mathias Drton, Bernd Sturmfels, and Seth Sullivant, \emph{Lectures on algebraic
  statistics}, Oberwolfach Seminars, vol.~39, Springer, Berlin, 2009, A
  Birkh\"{a}user book.

\bibitem{eisenbud96:_binom_ideal}
David Eisenbud and Bernd Sturmfels, \emph{Binomial ideals}, Duke Mathematical
  Journal \textbf{84} (1996), no.~1, 1--45.

\bibitem{TFP-II}
Alexander Engstr\"{o}m, Thomas Kahle, and Seth Sullivant, \emph{Multigraded
  commutative algebra of graph decompositions}, preprint (2011),
  arxiv:1102.2601.

\bibitem{geigermeeksturmfels06}
Dan Geiger, Christopher Meek, and Bernd Sturmfels, \emph{On the toric algebra
  of graphical models}, The Annals of Statistics \textbf{34} (2006), no.~5,
  1463--1492.

\bibitem{M2}
Daniel~R. Grayson and Michael~E. Stillman, \emph{{M}acaulay2, a software system
  for research in algebraic geometry}, Available at
  \url{http://www.math.uiuc.edu/Macaulay2/}.

\bibitem{HHHKR10:Binomial_Edge_Ideals}
J\"urgen Herzog, Takayuki Hibi, Freyja Hreinsd\'{o}ttir, Thomas Kahle, and
  Johannes Rauh, \emph{Binomial edge ideals and conditional independence
  statements}, Advances in Applied Mathematics \textbf{45} (2010), no.~3,
  317--333.

\bibitem{kahle11:binom-jsag}
Thomas Kahle, \emph{Decompositions of binomial ideals}, J. of Software for
  Algebra and Geometry \textbf{4} (2012), 1--5.

\bibitem{graphbinomials}
\bysame, \emph{\textsl{GraphBinomials}, a library for walks on graphs on
  monomials}, available from \url{https://github.com/tom111/GraphBinomials},
  2012.

\bibitem{kahle11mesoprimary}
Thomas Kahle and Ezra Miller, \emph{Decompositions of commutative monoid
  congruences and binomial ideals}, preprint (2011), arXiv:1107.4699.

\bibitem{lauritzen96}
Steffen~L. Lauritzen, \emph{Graphical models}, Oxford Statistical Science
  Series, Oxford University Press, 1996.

\bibitem{loera06:_markov_bases_of_three_way}
Jes\'{u}s A.~De Loera and Shmuel Onn, \emph{{M}arkov bases of three-way tables
  are arbitrarily complicated}, Journal of Symbolic Computation \textbf{41}
  (2006), 173--181.

\bibitem{malkin-2006}
Peter~N. Malkin, \emph{Truncated {M}arkov bases and {G}r{\"o}bner bases for
  integer programming}, 2006.

\bibitem{RauhAy2011:Robustness_and_CI_ideals}
Johannes Rauh and Nihat Ay, \emph{Robustness and conditional independence
  ideals}, preprint (2011), arxiv:1110.1338.

\bibitem{MBDB}
Johannes Rauh and Thomas Kahle, \emph{The {M}arkov bases database},
  http://markov-bases.de.

\bibitem{ReadWilson98:Atlas_of_Graphs}
Ronald~C. Read and Robin~J. Wilson, \emph{An atlas of graphs}, Clarendon Press,
  1998.

\bibitem{sturmfels91:_gr_bases_toric_variet}
Bernd Sturmfels, \emph{Gr{\"{o}}bner bases of toric varieties}, T{\={o}hoku
  Math. Journal} (1991), no.~43, 249--261.

\bibitem{sturmfels96:_gr_obner_bases_and_convex_polyt}
\bysame, \emph{{G}r\"{o}bner bases and convex polytopes}, University Lecture
  Series, vol.~8, American Mathematical Society, Providence, RI, 1996.

\bibitem{sturmfels02:_solvin_system_polyn_equat}
\bysame, \emph{Solving systems of polynomial equations}, CBMS, vol.~97,
  American Mathematical Society, 2002.

\bibitem{sullivant07:_toric}
Seth Sullivant, \emph{Toric fiber products}, J. Algebra \textbf{316} (2007),
  no.~2, 560--577.

\bibitem{swanson-taylor11_CI}
Irena Swanson and Amelia Taylor, \emph{Minimal primes of ideals arising from
  conditional independence statements}, preprint (2011), arXiv:1107.5604.

\end{thebibliography}

\end{document}